\documentclass{amsart}[11pt]
\usepackage{graphicx,url,amsthm,xspace,enumitem,subfig}
\usepackage[dvipsnames]{xcolor}
\usepackage{amscd}
\usepackage{multirow}
\usepackage[utf8]{inputenc}

\usepackage{amssymb}
\usepackage{graphics}
\usepackage{graphicx}
\usepackage{amscd}
\usepackage{xcolor}
\usepackage{amsmath, amsthm, epsfig,  psfrag}
\usepackage{pinlabel}
\usepackage{hyperref}
\usepackage[margin=3cm, marginpar=2.5cm]{geometry}

\usepackage[all,pdftex,arc,curve,color,frame]{xy}
\usepackage{xy}
\xyoption{all}

\usepackage{pinlabel}

\newcommand{\op}[1]{{\color{violet}{#1}}}
\newcommand{\ls}[1]{{\color{teal}{#1}}}

\newcommand{\C}{\mathbb{C}}
\newcommand{\R}{\mathbb{R}}

\newcommand{\Z}{\mathbb{Z}}
\newcommand{\cptwo}{\C\textup{P}^2}

\newcommand{\cptwobar}{\overline{\C\textup{P}}\,\!^2}

\newcommand{\ti}{\tilde}
\renewcommand{\d}{\partial}

\renewcommand{\div}{\operatorname{div}}

\DeclareMathOperator{\Int}{Int}
\DeclareMathOperator{\Diff}{Diff}

\newtheorem{theorem}{Theorem}[section]
\newtheorem{lemma}[theorem]{Lemma}
\newtheorem{prop}[theorem]{Proposition}
\newtheorem{cor}[theorem]{Corollary}

\theoremstyle{definition}
\newtheorem{definition}[theorem]{Definition}
\newtheorem{notation}[theorem]{Notation}
\newtheorem{remark}[theorem]{Remark}
\newtheorem{example}[theorem]{Example}

\numberwithin{equation}{section}

\title{Sandwiched singularities and nearly Lefschetz fibrations}

\author{Olga Plamenevskaya}
\address{Department of Mathematics, Stony Brook University, Stony Brook, NY,	11794,  U.S.A.}
\email{olga@math.stonybrook.edu}

\author{Laura Starkston}
\address{Deparment of Mathematics, University of California, Davis, 1 Shields Avenue, Davis, CA, 95616, U.S.A.}
\email{lstarkston@math.ucdavis.edu}
\begin{document}
	
	\begin{abstract}
	We study Milnor fibers and symplectic fillings of links of sandwiched singularities, with the goal of contrasting their algebro-geometric deformation theory and symplectic topology.   In the algebro-geometric setting, smoothings of sandwiched singularities are described by  de Jong--van Straten's theory: all Milnor fibers are generated from deformations of a singular plane curve germ associated to the surface singularity. We develop an analog of this theory in the symplectic setting, showing that all minimal symplectic fillings of the links are generated by certain immersed disk arrangements resembling de Jong--van Straten's picture deformations. This paper continues our previous work for a special subclass of singularities; the general case has additional difficulties and new features. The key new ingredient in the present paper is given by spinal open books and nearly Lefschetz fibrations: we use recent work of Min--Roy--Wang to understand symplectic fillings and encode them via multisections of certain Lefschetz fibrations.  As an application, we discuss arrangements that generate unexpected Stein fillings that are different from all Milnor fibers, showing that the links of a large class of sandwiched singularities admit unexpected fillings.    
\end{abstract}

	\maketitle
	
\section{Introduction}

In this paper, we compare and contrast deformation theory and symplectic topology of certain rational surface singularities. Extending our work in \cite{PS}, 
we study Stein fillings of the link of a singularity and compare them to Milnor fibers of possible smoothings. Let $X\subset \C^N$ be a singular complex surface with an isolated singularity at the origin. The intersection $Y= X \cap S^{2N-1}_{r}$ with the sphere of a small radius $r>0$ is 
a smooth 3-manifold called the {\em link of the singularity} $(X,0)$. The canonical contact structure $\xi$ on $Y$ is the distribution of complex tangencies to $Y$; 
the {\em contact link} $(Y, \xi)$ is independent of choices, up to contactomorphism. 
If  $(X, 0)$ admits a smoothing (that is, a deformation to a smooth surface, called a Milnor fiber), then the Milnor fiber carries a Stein structure and gives a Stein filling of $(Y, \xi)$.   For a  general surface singularity, the collection of Stein fillings is typically much larger than that of Milnor fibers, \cite{Akh-Ozbagci1,Akh-Ozbagci2, BNP}.   However, restrictions on the topology of the link often lead to rigidity phenomena, where it turns out that all possible Stein  fillings are generated by Milnor fibers. For example, when the link is a lens space, all minimal symplectic fillings are given by the Milnor fibers of the corresponding cyclic quotient singularity, \cite{Li,NPP-cycl}; analogous statements hold for certain other classes of singularities, \cite{OhtaOno1, Bhu-Ono,PPSU}. We will focus on sandwiched singularities, a subclass of rational singularities that can be defined by a combinatorial hypothesis on the dual resolution graph. The links of rational singularities are L-spaces in the Heegaard Floer sense \cite{Nem}; this, together with other properties, significantly restricts symplectic fillings. We will see that for sandwiched singularities, the symplectic topology of the fillings is still largely informed by the algebro-geometric constructions but rigidity begins to break down. This is the most interesting borderline case.

Inspiration for our constructions comes from de Jong--van Straten deformation theory for sandwiched singularities, \cite{DJVS}.  De Jong--van Straten showed that deformations of a sandwiched singularity $(X, 0)$ can be described via certain deformations of an associated decorated germ of a singular plane curve $(C, 0)$, and that the topology of the Milnor fiber of each smoothing can be easily read off from the corresponding deformed curve configuration. 
Indeed, the Milnor fiber is recovered by blowing up the (transverse) intersection points in the curve configuration in a Milnor ball and then taking the complement of the strict transforms of the components of the arrangement in the blow-up of the ball;  we explain the details in Section~\ref{sandwich-setup}.  In~\cite{PS}, we developed a symplectic analog of the de Jong--van Straten construction for a smaller subclass of rational singularities with reduced fundamental cycle: we showed that all minimal symplectic fillings of their links can be generated by certain smooth disk arrangements in a similar way. The two analogous constructions give insight into similarities and differences between Milnor fibers and more general Stein fillings. For certain links, even those as simple as Seifert fibered spaces, we found {\em unexpected} Stein fillings that are not homeomorphic to any of the Milnor fibers of singularities with the given link.

From the topological viewpoint, the key feature of rational singularities with reduced fundamental cycle is that the associated reducible singular curve germ $(C, 0)$ has {\em smooth} components. This enabled us to work with corresponding deformations at the level of smooth topology. In~\cite{PS}, we showed that each Milnor fiber with its Stein structure carries a compatible planar Lefschetz fibration: the fibration is the complement of a {neighborhood of a} collection of sections in a ``standard'' Lefschetz fibration on a blow-up of a ball in $\C^2$. The sections are related to the components of the germ $(C, 0)$: they are the strict transforms of the components of the deformed germ encoding the Milnor fiber in de Jong--van Straten's picture.  Moreover, each minimal symplectic filling arises 
from a similar planar Lefschetz fibration: this fact follows from Wendl's theorem describing  fillings of planar contact 3-manifolds \cite{We}, and we then encode the fibration via a curve configuration.

In this paper, we address the case of general sandwiched singularities. There are several difficulties compared to the previous case. The curve germ associated to the singularity now has cuspidal components, so one cannot immediately pass to smooth topology. From the symplectic viewpoint, the main difficulty is that the link $(Y, \xi)$ is no longer supported by a planar open book, so Wendl's theorem can no longer be used for classification of fillings. The key ingredient that solves the problem in this case is provided by spinal open books and nearly Lefschetz fibrations, \cite{LVHMW, LVHMW2, HRW}.  
We represent each Milnor fiber in de Jong--van Straten's construction as the complement of a collection of {\em multisections} in a certain standard planar Lefschetz fibration over the disk; a multisection of a fibration is a branched covering of the base. The complement of the multisections carries a {\em nearly Lefschetz fibration}, which has an  additional type of singular fiber along with the usual Lefschetz singularities. The link $(Y, \xi)$ is equipped with a planar {\em spinal} open book (unlike the classical case, its monodromy may interchange the boundary components of the page). We then use a recent result of Min--Roy--Wang \cite{HRW} generalizing Wendl's theorem: {every strong symplectic filling of a spinal planar open book is represented by a nearly Lefschetz fibration with the same planar fiber, compatible with the prescribed open book on the boundary. Our setting is simpler than the general setting of~\cite{HRW}, since the spine (generalization of binding) has very simple topology.
Since nearly Lefschetz fibrations can be thought of as complements of multisections, we can show that the data of the filling can then be encoded via a {\em DJVS immersed disk arrangement}  reminiscent of arrangements given by curve deformations in the de Jong--van Straten theory, see Definition~\ref{def:DJVSarrangement}.} We summarize our main result below, postponing the precise statement and all the details until Section~\ref{sandwich-setup}; see  Theorem~\ref{thm:maintechnical}.  {Note that while the classification results of~\cite{HRW} apply to strong fillings, all our contact $3$-manifolds are rational homology spheres, which erases the distinction between strong and weak fillings~\cite{OhtaOno3}. In all constructions, 
we produce Stein fillings.}

\begin{theorem} \label {thm1-intro} Let $(Y, \xi)$ be the contact link of a sandwiched singularity $(X, 0)$. 
{Then every compatible DJVS arrangement gives a Stein filling 
of $(Y, \xi)$},  
and all minimal weak symplectic fillings of $(Y, \xi)$ arise from the DJVS arrangements. Concretely, the filling is  the complement of {a neighborhood of} the strict transforms of the arrangement components in the appropriate blow-up of the 4-ball.    
\end{theorem}

Together with de Jong--van Straten theory, Theorem~\ref{thm1-intro} says that  Stein fillings of the link $(Y, \xi)$ and Milnor fibers of singularities in the given topological type can be reconstructed by the same procedure from  curve arrangements with similar properties. Thus, the difference between Milnor fibers and Stein fillings lies in the difference between the specific type of curve arrangements and their relation to the decorated germ encoding $(X, 0)$. For Milnor fibers, the curve arrangements are given by algebraic deformations of the decorated germ, whereas for Stein fillings they must only satisfy a certain boundary compatibility condition. Capturing this difference is quite subtle: it is difficult to tell whether a DJVS immersed disk arrangement is isotopic to an algebraic arrangement arising as a picture deformation. Even if we knew that two arrangements are not isotopic, 
it may still be possible that the corresponding Stein fillings are diffeomorphic (or even biholomorphic). Building on \cite{NPP} and on our previous work~\cite{PS}, we show in Section~\ref{sec:distinguish} that in certain cases, Stein fillings can be distinguished from Milnor fibers by the incidence matrix of the disk arrangement, and construct unexpected Stein fillings not homeomorphic to any of the Milnor fibers for possible smoothings. 
%Specific examples will be given in Section~\ref{sec:distinguish}. For now, we only state the general principle: \comm{decide whether we expand to more specific results in section 5 or modify these sentences}

\begin{theorem} \label{thm2-intro}  Many links of sandwiched singularities admit unexpected fillings.
In particular, for any $N>0$ and any sandwiched plumbing graph $G$, there is a sandwiched graph $K$ containing $G$ as a plumbing subgraph, such that the link of a singularity with the dual resolution graph $K$ admits at least 
$N$ pairwise non-homeomorphic unexpected Stein fillings.  
\end{theorem}

Unexpected fillings uncover the difference between the algebraic and symplectic settings. These are Stein domains 
that look quite similar to but still different from the ``standard'' fillings provided by Milnor fibers.  
As such, these constructions may have potential applications in 4-manifold topology; one could try to use unexpected fillings to produce interesting closed 4-manifolds. For example, in our next paper~\cite{BPS}
joint with M\'arton Beke, we were able to construct rational homology disk Stein fillings for certain links of singularities where rational homology disk smoothings do not exist.  
It would  be interesting to find other ways to construct further unexpected fillings and their further applications. 

{For studying classification questions and explicit constructions, it is important to encode nearly Lefschetz fibrations and multisections in a convenient language that allows us to prove topological theorems via combinatorial manipulations. For 
(nearly) Lefschetz fibrations, this language is given by monodromy factorizations. For DVJS arrangements, we use the language of braided wiring diagrams with vertical tangencies, 
generalizing the diagrams previously studied in the context of 
plane algebraic curves and line arrangements, \cite{Arvola, CS}.
A practical application is given by our constructions in~\cite{BPS}, based on the diagrammatic calculus for braided wiring diagrams that allows us to manipulate Stein fillings without changing the contact boundary. In Section~\ref{sec:diagrams}, we introduce the setup of braided wiring diagrams and give an algorithm for translating a diagram into the monodromy factorization of the corresponding nearly Lefschetz fibration, and vice versa.} 

The paper is organized as follows. Section~\ref{sandwich-setup} contains the main definitions and statements related to sandwiched singularities and the de Jong--van Straten theory. We define DJVS immersed disk arrangements and state a precise version of Theorem~\ref{thm1-intro}. 
{Section~\ref{sec:spinal} discusses spinal open books and nearly Lefschetz fibrations, explains how they arise in our setting, and uses the 
results of~\cite{HRW} to prove Theorem~\ref{thm1-intro}.
A significant part of the argument involves checking that certain spinal open books are compatible with the canonical contact structure on the links of the singularities that we study.  
In Section~\ref{sec:diagrams}, we develop a dictionary for translating between DJVS disk arrangements and monodromy factorizations of nearly Lefschetz fibrations.} Section~\ref{sec:distinguish} explains how to distinguish the fillings using the incidence data; we summarize our previous results on unexpected fillings, prove Theorem~\ref{thm2-intro}, and include a discussion of deformation (non)realizability of arrangements.   We have made an effort to include sufficient background and details to make the paper reasonably self-contained.

{\bf Acknowledgements:} We are grateful to M\'arton Beke, Andras N\'emethi, and Mark Powell
for helpful conversations, and especially to Agniva Roy and Luya Wang for explaining spinal open books to us. We thank Stepan Orevkov for very helpful discussions of deformation realizability of curve arrangements. 
Some of the initial discussions motivating this work took place at the conference ``New structures in low-dimensional topology'' in Budapest, July 2024 and at the workshop ``What's your trick?'' in Banff, August 2024. The authors discussed the final stages of this project during the Georgia Topology Conference in May 2025. OP has been partially supported by the NSF grant DMS 2304080. 
LS has been partially supported by NSF CAREER grant DMS 2042345 and Sloan grant FG 2021-16254.

\section{Sandwiched singularities, de Jong--van Straten's theory, and smooth disk arrangements} 
\label{sandwich-setup}

\subsection{Decorated germs and picture deformations} 

Let $(X, 0)$ be a normal surface singularity. For a resolution $\pi: \ti{X} \to X$, the {\em exceptional divisor} $\pi^{-1} (0)$ is the inverse image of the singular point. Performing additional blow-ups if necessary, we can assume that 
the exceptional divisor $\pi^{-1} (0)$ has normal crossings. This means that $\pi^{-1} (0)= \cup_{v \in G} E_v$, where
the irreducible components $E_v$ are smooth complex curves that intersect transversally at double points only. 
A resolution with this property is  called a {\em good} resolution; 
a minimal good resolution is unique. 
(Minimality means that $\ti X$ contains no no embedded smooth complex curves of genus 0 and self-intersection $-1$ such that the resolution would be still be good after blowing down the curve.)
The topology of a {good} resolution is encoded by the dual resolution graph $G$.  The vertices  $v\in G$ correspond to the exceptional curves $E_v$, and the edges correspond to intersections of the irreducible components. In this paper, all singularities will be rational; then all exceptional curves have genus 0, and $G$ is a tree. 

Sandwiched singularities can be defined by a combinatorial hypothesis on the dual resolution graph. (For simplicity, we always work with the minimal good resolution.) The graph $G$ is {\em sandwiched}  if it can be augmented by adding new valency one vertices with self-intersection $(-1)$  in a way that the resulting augmented graph $G'$ can be blown to a smooth point, see Figure~\ref{fig:examplesandwich} for an example. (The choice of such augmentation is not generally unique.) In other words, for a sandwiched singularity 
$(X, 0)$ there exists an embedding of the tubular neighborhood of the exceptional set of the resolution $\tilde{X}$ into some blow-up of $\C^2$.  The $(-1)$ vertices correspond to a distinguished collection of $(-1)$ curves in this space,  so that the configuration of these $(-1)$ curves 
together with the original exceptional set can be completely blown down. It is not hard to show that a minimal good sandwiched graph $G$ cannot have any $(-1)$ vertices. It follows that the minimal good resolution is the same as the minimal resolution in the sandwiched case, and also that all $(-1)$ vertices in the augmented graph $G'$ come from the augmentation.  

\begin{figure}
	\centering
	\includegraphics[scale=.5]{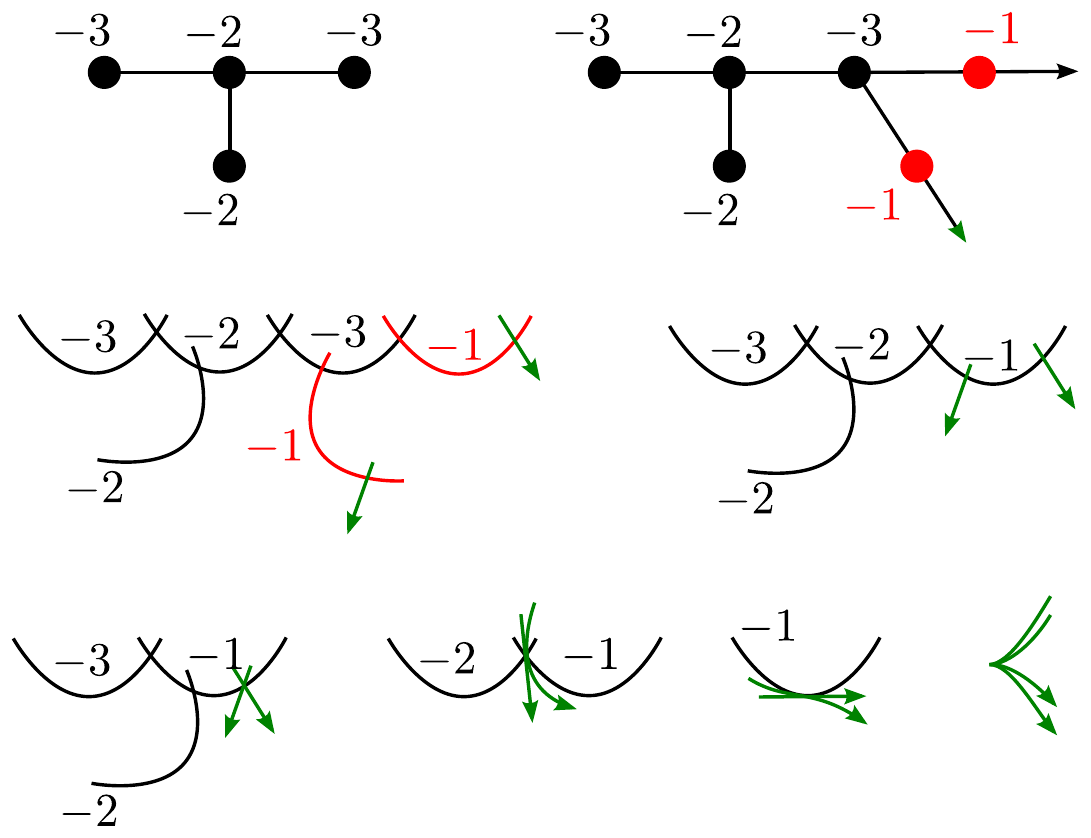}
	\caption{A resolution graph of a sandwiched singularity together with an augmentation with $-1$ curves, decorated by transverse disks. Below is the sequence of blow-downs and the image of the disks under these blow-downs. In this case, the resulting curve configuration consists of two simple cusps intersecting each other with multiplicity $7$. By tracking the multiplicities of the exceptional spheres at each stage, one can see that the weight $w_i$ of each component is $6$.}
	\label{fig:examplesandwich}
\end{figure}

For a sandwiched $(X, 0)$, the embedding of its minimal (good) resolution together with the additional $(-1)$ curves into a blow-up of $\C^2$  can be encoded by a (generally reducible) singular plane curve germ, as follows. 
The decorated germ is topologically unique for each fixed choice of the augmented graph $G'$, but it depends on this choice.
For each  distinguished
$(-1)$ curve, fix  a transverse complex disk $\tilde{C}_i$ through a generic point. These disks, as well as their images under blow-downs, are called {\em curvettas} in \cite{DJVS}.   When we contract the the curve configuration
corresponding to the augmented graph $G'$, the union of these disks becomes a germ of a curve in $\C^2$, with 
irreducible components $C_i$ given by the images of the individual disks $\ti{C}_i$ under the blow-downs. Since each $\ti{C}_i$ intersects some exceptional curve, each component $C_i$ passes through $0$. We consider 
the germ of $C= \cup C_i$ at 0. The germ comes with additional data: we
define an integer weight $w_i = w(C_i)$ for each $C_i$. The weight $w_i$ is the sum of the intersection numbers $w_{i,j}$ of the images of $\ti{C}_i$ with 
the $(-1)$ curves as we go through the blow-down process for $G'$. At each stage, if we blow down the $(-1)$ curve $E_j$, $w_{i,j}$ is the intersection multiplicity of $E_j$ and the image of  $\ti{C}_i$ at this stage. Equivalently, after $E_j$ is blown down to a point $q_j$, $w_{i,j}$ is the multiplicity of the subsequent image of $\ti{C}_i$  at $q_j$. 
In the special case where the component $C_i$ is smooth after the final blow-down, each intersection multiplicity $w_{i,j}$ is $0$ or $1$, and $w_i$ is the number of blow-downs that the corresponding component goes through. Let $w$ denote the collection of weights $w_i$. Then $(C, w)$ is the {\em decorated germ} 
encoding the singularity $(X, 0)$. See Figure~\ref{fig:examplesandwich} for an example.

\begin{remark}\label{rmk:-2legs} The following observation illuminates the role of the decoration by weights. Given a graph $G$ and its augmentation $G'$, consider a new graph $K$ where each $(-1)$ vertex of $G'$ is replaced by a linear chain of $(-2)$ vertices, see Figure~\ref{fig:-2legs}.  The graph $K$ is sandwiched; an augmented graph $K'$ that fully blows down can be obtained by adding a $(-1)$ vertex at the end of each $(-2)$ chain. (For the analytic data defining the singularity, consider the configuration of curves associated to $K'$ as a further blowup on $G'$.)   
The resulting plane curve germ $C_K$ is then identical to the original germ $C_G$, but the weights for $C_K$ are higher than those for $C_G$: each weight increases by the length of the corresponding $(-2)$ chain. 
\end{remark}

\begin{figure}
	\centering
	\includegraphics[scale=.4]{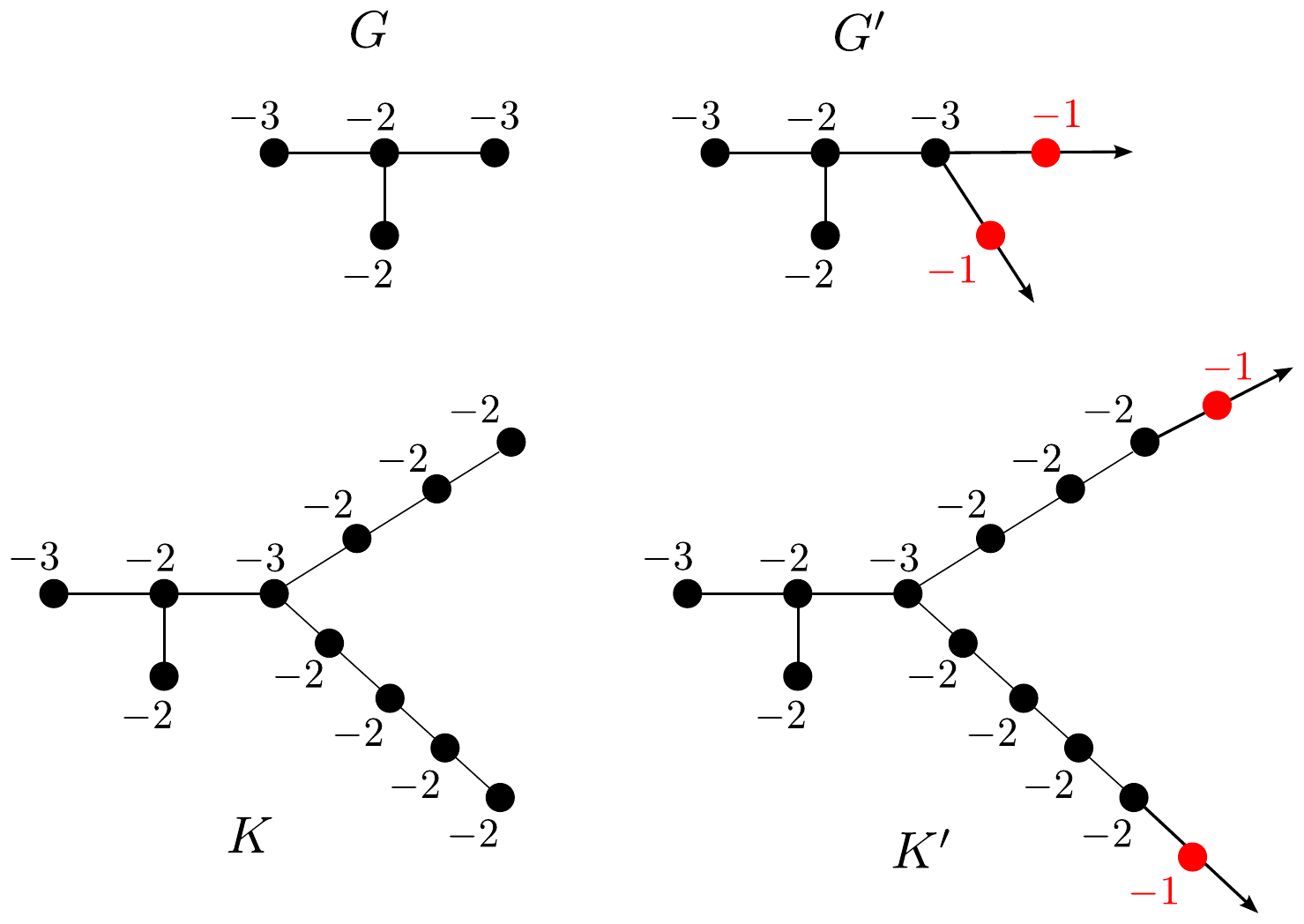}
	\caption{Top: A resolution graph $G$ of a sandwiched singularity together with an augmentation. Bottom: A larger sandwiched resolution graph $K$ with an augmentation whose (non-decorated) plane curve germ is the same as $C_G$,  but the weights are higher. As in Figure~\ref{fig:examplesandwich}, the weights associated to the curve components of $C_G$ with the chosen augmentation $G'$ are both $8$. By adding the additional legs of $-2$ spheres, the weights of the two curve components $C_K$ (with the indicated augmentation $K'$) are $11$ and $12$.}
	\label{fig:-2legs}
\end{figure}

It will be useful to label the vertices of $G$ adjacent to the additional $(-1)$ curves and the disks $\ti{C}_i$.

\begin{notation} \label{def:EC} 
Let $E(C_i)$ denote the vertex in $G$  such that the disk $\ti{C}_i$ intersects a $(-1)$ curve attached to $E(C_i)$ in the augmented graph $G'$. For brevity, we will say that the curvetta $C_i$ sits on $E(C_i)$. 
Let $E_0$ be the last surviving vertex of $G$ under the blow-down sequence for the graph $G'$. 
\end{notation}

The original singularity $(X, 0)$ can be reconstructed from $(C, w)$ via a sequence of blow-ups at 0 and infinitely near points: we recover the augmented graph $G'$ and the embedding of the tubular neighborhood of the exceptional set of the resolution given by $G$ into some blow-up of $\C^2$. The disks $\ti{C}_i$ are the strict transforms of the curvettas $C_i$.  The vertex $E_0 \in G$  corresponds to the exceptional sphere of the first blow-up of the sequence.

De Jong--van Straten's theory describes deformations of a sandwiched surface singularity via certain deformations of its decorated germ. We only state the main result of \cite{DJVS} for smoothings.

\begin{definition} \label{def:picturedef} 

A {\em picture deformation} of a decorated germ $C= \cup_i C_i$ with weights $w_i$  is a 1-parameter small analytic deformation $C^t$ equipped with a collection $p$ of marked points (for each $t$) such that

(i)    $C^t$ is a $\delta$-constant deformation,

(ii)   the only singularities of the curve $C^t$ are transverse multiple points for all $t>0$,

(iii) the marked points $p_j$ are chosen on $C^t$, so that each intersection between two irredicible components is marked; additionally, there may be free marked points on each
$C_i$ away from the intersections,

{(iv) the sum of multiplicities of the marked points on each component $C_i^t$ 
equals $w_i$.}

{More precisely, the collection of marked points $p$ should be thought of as deformation of the corresponding scheme originally concentrated at 0, see\cite{DJVS}, but this will not be important for our purposes.}
\end{definition}

Note that every decorated germ admits at least one picture deformation, see section~\ref{ss:Scott} which describes the Scott deformation.

\begin{remark}
	Note that a small analytic deformation $C^t$ means that for all $t\neq 0$, the curves $C^t$ are equisingular. This is different than a smooth 1-parameter family of curves $C^t$ which might change their singularity types generically at isolated times. As discussed in~\cite[Section 8]{PS}, this key difference between smooth 1-parameter families (which can change the singularity types multiple times) and deformations (which keep the same singularity types immediately after deforming the germ) seems to be a primary contributor to differences between the algebraic deformation theory and symplectic fillings.
\end{remark}

To see the Milnor fibers, we fix a good representative of the curve germ in a Milnor ball $B$; our construction will be carried out in this fixed ball.
Because the deformation is  $\delta$-constant, it follows from the existence of simultaneous normalization \cite{Tes} that the deformation preserves the branches, and each $C^t_i$ is an immersed disk in $B$. {We can write $C^t_i=n^t_i(D)$, where $D$ is the complex disk, and $n^t_i: D \to B$ are the normalization maps.}

  \begin{theorem} \cite[Theorem 4.4, Lemma 4.7]{DJVS} Let $(X, 0)$ be a sandwiched singularity and fix a decorated germ $(C, w)$ encoding $(X, 0)$.   Every smoothing of $(X, 0)$ arises from one of the picture deformations of $(C, w)$. The Milnor fiber of a smoothing can be reconstructed from the picture deformation as follows: blow up $B$ at all the marked points of the deformed configuration $C^t$ and take the complement of the strict transforms $\tilde{C}^t_i$ of the components $C_i$ in $B \#_n \cptwobar$.
   \end{theorem}
 Since all intersections are transverse, and we blow up at the intersection points, the strict transforms of the components $C_i$ are disjoint smooth curves. We will be working with the Milnor fiber as a compact domain with boundary, 
 \begin{equation} \label{eq:Milnorfiber}
 W = [B  \#_n \cptwobar] \setminus \cup_i \nu(\tilde{C}^t_i),
 \end{equation}
 where $B$  stands for a closed Milnor ball.

 We fix a choice of coordinates $(x, y)$ on $\C^2$ such that that none of the tangent cones of the branches $C_i$ of the decorated germ are vertical. 
 We will think of the Milnor ball as the product $B = D_x \times D_y$, with corners smoothed. We can assume that the projection $C_i \to D_x$ is a branched covering whose degree is equal to the multiplicity of $C_i$ at $0$, with a single branch point at the origin. Further, with a generic choice of coordinates we can also assume that for every picture deformation $C^t$ we consider, all vertical tangencies of the projections $\pi_x: C^t_i \to D_x$ are non-degenerate, that is, outside of self-intersections, each projection is a branched covering with simple branch points only. {If we consider the normalization maps as above, then $\pi_x \circ n^t_i: D \to D_x$ is a simple branched covering for each $i$.}

 \subsection{DJVS disk arrangements}
 
 This paper continues our study of sympectic generalizations of the de Jong--van Straten theory. As in \cite{PS}, we will be working with more general arrangements instead of picture deformations.  As described above, with the appropriate choice of coordinates where each irreducible component $C_i^t$ 
is a simple branched covering of $D_x$, the curve arrangement $C_t$ for $t\neq 0$ from a picture deformation is a special case of a DJVS immersed disk arrangement.

\begin{definition} \label{def:DJVSarrangement}
Let 
$(\Gamma, p)$ be an arrangement of immersed smooth disks  
$\Gamma =\cup \Gamma_i$ in $D_x \times D_y$, {with  $\Gamma_i= n_i(D)$ for an immersion 
$n_i: D \to D_x \times D_y$. Let $p=\{p_j\} \subset \Gamma$ be a finite collection of marked points.}

We will say that $(\Gamma, p)$ is 
a {\em DJVS immersed disk arrangement} if it satisfies the following properties:

{(i) $\pi_x \circ n_i: D\to D_x$ is a simple branched covering for each $i$;}

(ii) All intersections $\Gamma_i \cap \Gamma_k$ and self-intersections of 
the components $\Gamma_i$ are positive transverse multiple points locally modelled on the intersection of complex lines.

(iii) All intersections are marked, and there can be additional free marked points.

(iv) Each disk $\Gamma_i$ intersects the boundary of $D_x \times D_y$ transversely, $\Gamma_i\cap \partial(D_x \times D_y) \subset \partial D_x \times \Int D_y$, all marked points and branch points are contained in the interior of $D_x \times D_y$ and their projections to $D_x$ are distinct.

The {\em weight} $w(\Gamma_i)$ of the component $\Gamma_i$ is the sum of multiplicities of the points $p_j$ on $\Gamma_i$.

We will always assume that $\Gamma_i$ is oriented compatibly with the orientation of $D_x \subset \C$.  
 \end{definition}

\begin{definition} We say that a DJVS arrangement $(\Gamma, p)$ has the same weights and boundary data as a decorated germ $(C, w)$, or simply that  \emph{$(\Gamma, p)$ is compatible with $(C, w)$}, if there is a bijection between the components of $C$ and $\Gamma$,  $\d \Gamma_i= \d C_i$, and $w(\Gamma_i)= w_i$.  
	%In other words, compatibility means that the DJVS arrangement has the same boundary data and the same weights as the decorated germ.
 \end{definition}

 Note that if $(C^t, p)$ is a picture deformation of $(C, w)$, then, strictly speaking, the boundary of $C^t$ is different from (but isotopic to) the boundary of the singular curve $C$. We always assume that the components of the picture deformation are moved by a small isotopy supported near the boundary (away from all the intersection points), so that  
 $(C^t, p)$ is compatible with $(C, w)$.
 
Together with Milnor fibers constructed as in \eqref{eq:Milnorfiber}, we will consider 4-manifolds built from DJVS arrangements in a similar way.  If $(\Gamma, p)$ is a DJVS immersed disk arrangement, let
\begin{equation} \label{eq:DJVSfilling}
W_{(\Gamma,p)} = [(D_x \times D_y)  \#_n \cptwobar] \setminus \cup_i \nu(\tilde{\Gamma}_i),
\end{equation}
be the complement of the union of small tubular neighborhoods of the (disjoint) strict transforms $\tilde{\Gamma}_i$ of the  disks $\Gamma_i$ in the blow-up of $D_x \times D_y$ at all the marked points $p_j$, $j=1, \dots, n$. Here, the strict transforms $\ti{\Gamma}_i$  
can be locally modelled on the algebro-geometric strict transforms due to hypothesis~(ii). We will sometimes say that 
$W_{(\Gamma, p)}$ is obtained from $(\Gamma, p)$ by {\em the DJVS construction}. 

The following theorem explains the significance of DJVS arrangements and gives a more precise statement of Theorem~\ref{thm1-intro}.

\begin{theorem} \label{thm:maintechnical} Let $(X, 0)$ be a sandwiched singularity with the germ $(C, w)$ and the contact link $(Y, \xi)$.

(1) If $(\Gamma, p)$ has the same boundary data and the same weights as 
a decorated germ $(C, w)$, then $W_{(\Gamma, p)}$ is a Stein filling of the link $(Y, \xi)$.

(2) Every minimal symplectic filling of $(Y, \xi)$ is symplectic deformation equivalent to $W_{(\Gamma, p)}$ for some DJVS arrangement $(\Gamma, p)$ compatible with $(C, w)$.
\end{theorem}

The proof will be given in the next section.

\begin{remark} \label{rmk:top-vs-an}
We need to point out an important subtle point regarding analytic vs.~topological data. 
For a normal surface singularity $(X, 0)$, its  minimal normal crossing resolution of $(X, 0)$ gives a smooth 4-manifold $\tilde{X}$ which is the plumbing of disk bundles according to the (negative definite) dual resolution graph $G$.  The link $Y = \d \tilde{X}$ is the boundary of this plumbing. By~\cite{Neu}, the links of two normal surface singularities have the same oriented diffeomorphism type if and only if their minimal good resolutions have the same dual graphs.  
The local topological type of the singularity $(X, 0)$ can be understood from its link $Y$, as a cone on the corresponding 3-manifold. Two singularities are {\em topologically equivalent} if they have the same link. It is important to note that the analytic type of the singularity is not uniquely determined by the link; typically, infinitely many analytically different singularities have diffeomorphic links.

It turns out that the canonical contact structures are all isomorphic for different surface singularities of the same topological type~\cite{CNPP}; thus, the dual resolution graph encodes the canonical contact structure. Indeed, 
this contact structure can be recovered as the convex boundary of the plumbing, according to the graph, of the standard neighborhoods of the corresponding symplectic surfaces~\cite{ParkStipsicz}. It follows that Milnor fibers of smoothings of all analytic singularities  with the same link $(Y, \xi)$ provide Stein fillings of the link.  This implies that if we wish to study Stein fillings arising as Milnor fibers, we need to consider all possible analytic singularities of the given topological type. We say that a Stein filling of $(Y, \xi)$ is {\em unexpected} if it is not diffeomorphic to any of the Milnor fibers of {\em any} of these singularities.  (For a general normal surface singularity, an additional Stein filling is provided by a deformation of the minimal resolution \cite{BO}, but since for rational singularities the minimal resolution  is diffeomorphic to the Milnor fiber of the Artin smoothing, we do not need to consider the resolution separately.) 
 
 A decorated germ $(C, w)$ encodes the analytic type of a sandwiched singularity $(X, 0)$. However, if $(X', 0)$ is topologically equivalent to $(X, 0)$, then it has the same dual resolution graph. The same combinatorial choice of the sandwiched augmentation  produces a decorated germ $(C', w)$ where the curve germ $C'$  is topologically equivalent to $C$, and the weights of the two germs are the same. Since the topological type of a singular plane curve germ is determined by its boundary link, it follows that any DJVS arrangement that is compatible with $(C, w)$ is also compatible with $(C', w)$ after an isotopy of its boundary.  
 \end{remark}

\subsection{The cap construction} Several arguments will require a particular {\em cap} $U$ of a filling  $W=W_{(\Gamma, p)}$ as in \eqref{eq:DJVSfilling}, such that $W \cup U$ is a blow-up of a 4-sphere, \cite{NPP, PS}.  We fix a sandwiched singularity with a 
decorated germ $(C, w)$ and consider DVJS arrangements with the same boundary data and weights as this germ. 
The cap $U$ will be determined by $(C, w)$ and thus will be the same for all fillings 
$W_{(\Gamma, p)}$ such that 
$(\Gamma, p)$ is compatible with $(C, w)$.
To construct $U$, let $B \subset \C^2$ be a closed Milnor ball as above. We assume that 
$B$ contains both the branches of the germ $\mathcal{C}$ and the arrangement $\Gamma$.  Let $(B', \mathcal{C'})$ be another copy of this ball with 
the  germ $\mathcal{C}$ inside, {with reversed orientation.}
%After an isotopy of the boundaries of the disks $\Gamma_i$ to match $\partial C_i$, 
We can glue 
$(B, \Gamma)$ and $({B}', {\mathcal{C}}')$ so that the boundary of $\Gamma_i$ is glued to the boundary of the corresponding germ branch ${C}'_i$. With our choice of orientations, the result of gluing 
is a smooth 4-sphere $B \cup {B}'$ containing the 2-spheres $\Sigma_i=\Gamma_i \cup {C}_i'$. The sphere $\Sigma_i=\Gamma_i \cup {C}_i'$ is smoothly immersed with transverse self-intersections on the $\Gamma_i$ side and may have a singular point at the origin in $B'$ if 
${C}_i'$ has one. Blowing up at the points $p_1, \dots, p_n\in B$, we get
$\#_{i=1}^n \cptwobar$, represented as the blow-up $\tilde{B}$ of the ball $B$
glued to ${B}'$. The strict transforms 
$\tilde{\Sigma}_i= \tilde{\Gamma}_i \cup C'_i$  are topologically embedded spheres that are smoothly embedded, except possibly at the origin in $B'$. 
Let $T_i$ be a thin tubular neighborhood of the strict transform of $\tilde{\Gamma}_i$ in $\tilde{B}$. By construction,  
$W_{(\Gamma, p)} = \tilde B \setminus \cup_{i=1}^m T_i$. Setting $U = {B}' \cup_{i=1}^m T_i$, 
we have $U \cup W = \#_{i=1}^n \cptwobar$.

As in \cite[Lemmas  4.2.1, 4.2.4]{NPP}, we see that the cap $U$ is determined by the weights and the boundary link  
$\cup_i  \d \Gamma_i$ of the arrangement $\Gamma$.  
Indeed, $U$ is the  
4-ball ${B}'$ with 2-handles attached along the boundaries of the $C_i$'s. The framing of the handles is given by the self-intersection of the corresponding surfaces $\tilde{\Sigma}_i$, which is computed in \cite[Lemma  4.2.1]{NPP} to be 
\begin{equation}\label{eq:Sigma-self-int}
\tilde{\Sigma}_i \cdot \tilde{\Sigma}_i = - w_i - 2 \delta_i,
\end{equation}
where $\delta_i= \delta(C_i)$ is the $\delta$-invariant of $C_i$. (The argument in \cite{NPP} is stated for complex curves but uses smooth topology and goes through in our setting because $\tilde{\Sigma}_i$ agrees with the complex curve $C_i$ near the only singular point of $\tilde{\Sigma}_i$.)
It follows that the cap will be the same for all the DJVS arrangements that have the same boundary and the same weights as a fixed decorated germ $(C, w)$.

\subsection{Topology of plumbed 3-manifolds}\label{sec:top-plumbing}
We need a brief detour to discuss plumbed 3-manifolds and their mapping class groups. 
 
For a normal surface singularity $(X, 0)$, its minimal good resolution is, as a smooth 4-manifold with boundary, a plumbing of disk bundles over surfaces (exceptional curves), according to its dual resolution graph. The graph is a tree, and all exceptional curves are spheres if the singularity is rational.
 As the boundary of a plumbing of disk bundles, the link $Y$  of $(X, 0)$ carries a {\em plumbing structure}: a family of pairwise disjoint embedded tori whose complement is fibered by circles, together with the fibration structure on the complement, so that on each torus, the intersection number of the fibers from each side is $\pm 1$.  The tori correspond to the edges of the dual graph; the fibered connected components of their complement correspond to the vertices. We will refer to these connected components as {\em pieces} of~$Y$. 
 
 We can recover the normal Euler numbers of the disk bundles labeling the dual graph from the fibrations in the plumbing structure, as follows. Each piece $P$ is a circle bundle over an orientable surface with boundary $\Sigma$, thus the fibration structure identifies $P$ with the trivial bundle $\Sigma\times S^1\to \Sigma$. For each boundary component $T_j=\partial_j\Sigma\times S^1$, identify a basis for $H_1(T_j)\cong \Z^2$ by setting $\partial_j\Sigma \times \{\theta_0\} = (1,0)$ and $\{p\}\times S^1 = (0,1)$. Each $T_j$ has another fibration structure coming from the piece on the other side of it. Denote the class of the fiber in this other fibration structure by $A_j\in H_1(T_j)$. The plumbing structure definition requires that $A_j\cdot (0,1) = \pm 1$. Since the fibers are not canonically oriented, we may assume by switching the orientation on $A_j$ that this is $+1$ so $A_j = (1,n_j)$ for some $n_j\in \Z$. To recover the normal Euler numbers labeling the dual graph, we want to ''undo'' the plumbing to see the circle bundle over the closed surface. Since plumbing two disk bundles glues the fibers of one bundle to the section of the other and vice versa, to undo the plumbing, we glue to each $T_j$ a solid torus such that its meridian is glued to the class $A_j=(1,n_j)$. 
 From here, we can see that the normal Euler number is $\sum_j -n_j$. For a Dehn surgery description of why this is true, specialize to the case where every disk bundle is over a $2$-sphere for simplicity (which is always true in our setting of rational singularities). The trivial $S^1$ bundle over $S^2$ is obtained from $S^3$ as $0$-framed Dehn surgery on the unknot. Meridians $F_1,\dots, F_r$ of the unknot represent fibers of this trivial bundle. Deleting solid torus neighborhoods $F_1,\dots, F_r$ yields $\Sigma\times S^1$. Gluing a solid torus back in with the meridian of the solid torus identified with the class $(1,n_j)$ (i.e. $1$ time in the section direction which is represented by a meridian of $F_j$ and $n_j$ times in the fiber direction represented by a longitude) corresponds to $1/n_j$ Dehn surgery on $F_j$. Thus, the circle bundle is given by Figure~\ref{fig:DSbundle} which is equivalent by Rolfsen twists/slam dunks to $\sum_j -n_j$ integral surgery on the unknot~\cite[Section 5.3]{GS}, which is the boundary of the disk bundle over the sphere of normal Euler number $\sum_j -n_j$.
 
 \begin{figure}
 	\centering
 	\includegraphics[scale=.5]{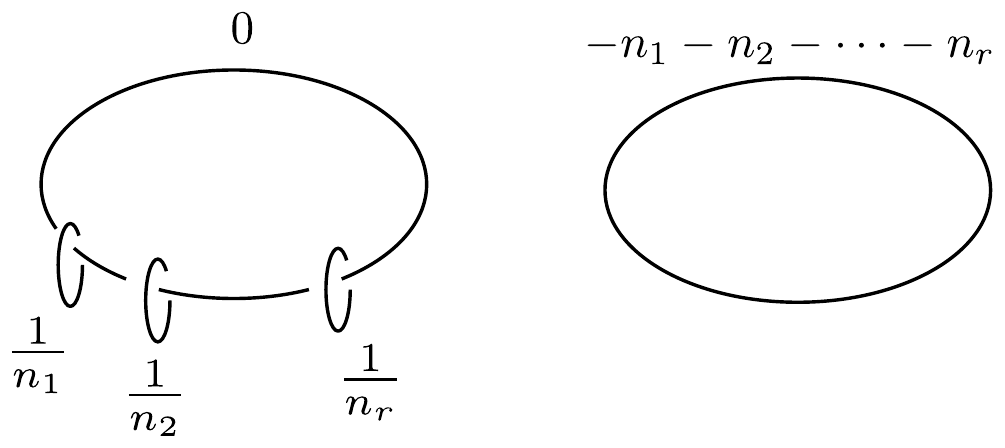}
 	\caption{Dehn surgery diagram to recover normal Euler numbers in resolution graph from plumbing structure fibrations.}
 	\label{fig:DSbundle}
 \end{figure}
  
 The plumbing structure corresponding to the minimal good resolution is determined by the oriented link $Y$; it is invariant, up to isotopy, by all orientation-preserving diffeomorphisms of the boundary. This result was proved by Popescu-Pampu~\cite{PP} building on similar classical facts such as uniqueness of the JSJ decomposition. We briefly review this material since the implications for the mapping class groups will be important to us. Recall

 \begin{theorem}[JSJ decomposition~\cite{JacoShalen,Johannson}]\label{thm:JSJ}
		Given a closed irreducible 3-manifold, there is a unique (up to isotopy) minimal collection of disjoint incompressible tori such that each component of the complement of this collection of tori is either Seifert fibered or atoroidal.
	\end{theorem}
	
	Because the minimal set of incompressible tori is unique up to isotopy, any self-diffeomorphism of a $3$-manifold $Y$ must preserve the decomposition. In other words, each JSJ component in the complement of the incompressible tori must be sent by the diffeomorphism to another (diffeomorphic) JSJ component.
	
	We focus on JSJ decompositions for graph manifolds with plumbing graphs which are trees, and each vertex corresponds to a disk bundle over a sphere. In the plumbing graph, each edge corresponds to a torus in the 3-dimensional graph manifold (this is the $S^1\times S^1$ on the corner of the $D^2\times D^2$ used for the plumbing of the disk bundles). Note that if there is a linear chain of edges so that all vertices in the middle of the chain are valency $2$, then the tori corresponding to any two edges in the chain are isotopic to each other (such a linear chain corresponds to a $T^2 \times I$ subset of the graph manifold). A linear chain ending in a vertex of valency $1$ has edges corresponding to compressible tori (since such a linear chain corresponds to a solid torus in the graph manifold). Finally, a connected subset of the graph which contains at most one vertex of valency greater than $2$ corresponds to a Seifert fibered portion of the graph manifold. Thus the minimal set of incompressible tori yielding the JSJ decomposition corresponds to a minimal set of edges cutting the graph into a disjoint union of star-shaped graphs, and each JSJ component is Seifert fibered.

 The collection of tori which define a plumbing structure on a graph manifold is generally larger than the JSJ collection of tori, so Popescu-Pampu's result on uniqueness of plumbing structures refines the uniqueness of JSJ decompositions for graph manifolds. We briefly summarize the idea of Popescu-Pampu's argument. 
 He first adds additional tori to the JSJ decomposition so that each piece in the complement is genuinely fibered (no exceptional fibers)--this comes from Waldhausen's notion of a minimal graph structure. The fibration structure of each piece is determined by the topology of the piece, unless the piece is a solid torus. This uses that the base of each plumbing vertex is orientable, and follows from the fact that $T^2\times I$ and the solid torus are the only 3-manifolds admitting non-unique Seifert structures over an orientable base. Then the decomposition is refined further to split the linear chains with more than one edge and to determine the plumbing structure on the solid tori. The main difference between a graph structure and a plumbing structure is that in a plumbing structure, when two fibrations are glued along a torus, the homology classes of the fibers from the two sides must have intersection number $\pm 1$ in the gluing torus. In a graph structure, the only requirement is that the fibers are not parallel. Popescu-Pampu uses a combinatorial lattice procedure to uniquely break down $T^2\times I$ with two fixed fibrations on the two boundary tori into a sequence of $T^2\times I$ pieces with distinct fibrations on each copy, starting and ending with the chosen fixed fibrations on the boundary but meeting at intermediate tori according to the plumbing structure requirement. The uniqueness here uses negative definiteness and minimality of the plumbing graph. In solid tori, this is done in the same way by interpolating between the fibration of $T^2$ coming from the other side on the boundary torus and the fibration of $T^2$ by meridians. The piece corresponding to the valence 1 vertex in a linear chain has fibration structure coming from the last $T^2\times I$ piece in the interpolation, filled in to a solid torus by matching the meridians. Thus we have

	\begin{theorem}[Theorem 9.7~\cite{PP}]
	The plumbing structure induced by the minimal normal crossing resolution of the link $Y$ of a normal surface singularity $(X, 0)$ is invariant under $\Diff^+(Y)$, up to isotopy.
	\end{theorem}
	
	This means that any orientation-preserving diffeomorphism preserves (up to isotopy) the set of cutting tori corresponding to edges of the graph as well as the fibration structures on each piece corresponding to the vertices of the graph. 
	As described above, the normal Euler number of the bundle is determined by the fibrations of the plumbing structure. In other words, any orientation-preserving diffeomorphism corresponds to an automorphism of the plumbing graph (with vertices decorated with their self-intersection numbers). In particular, every such diffeomorphism is isotopic to the identity if the decorated plumbing graph has no symmetries.
 
 \subsection{A marking of the plumbing} 
 
 A decorated germ $(C, w)$ of a sandwiched singularity $(X, 0)$ induces a {\em marking} of its plumbing structure, \cite{NPP}.  
	Suppose that the irreducible components $\{C_i\}_{i=1}^m$  of $C$ are labeled by the integers $i=1, \dots m$. The marking is the map
	 from the set $\{1, 2, \dots, m\}$ to the set of pieces of $Y$ in the plumbing structure: if the curvetta $C_i$ sits on the exceptional curve $E(C_i)$ as in Notation~\ref{def:EC}, then this map sends $i$ to the piece of $Y$ corresponding to the vertex $E(C_i)$. This is equivalent to a labeled collection of fibers $f_1,\dots, f_m$ of the plumbing structure where $f_i$ is a fiber of the piece corresponding to $E(C_i)$, so we can view the marking as a labeled subset of $Y$.
	 
	 \begin{cor} \label{cor:diff} Suppose that the plumbing graph corresponding to the decorated germ $(C, w)$ admits no non-trivial automorphisms. Then the marking is preserved by any orientation-preserving diffeomorphism of the link $Y$.
	 \end{cor}

	 In this sense, once we fix a decorated germ encoding $(X, 0)$, the marking becomes part of the intrinsic data of the link, as long as there are no symmetries.  
 
\section{Spinal open books and nearly Lefschetz fibrations} \label{sec:spinal}

In \cite{PS}, we used open books and Lefschetz fibrations to build a symplectic analog of the de Jong--van Straten theory in the reduced fundamental cycle case, describing Stein fillings of the contact link $(Y, \xi_{can})$ of $(X, 0)$ via certain curve arrangements. To extend the story to the general sandwiched case, we    
will use spinal open books and nearly Lefschetz fibrations developed in \cite{LVHMW, LVHMW2, HRW}.   

\subsection{Spinal open books on links of singularities}
  
A classical open book decomposition of a 3-manifold $Y$ is given by a fibered link $\beta$ in $Y$, called the {\em binding}, and a surface bundle $M$ over $S^1$ whose fibers are connected oriented surfaces with boundary called {\em pages}, so that $\beta$ is the boundary of each page. The monodromy of the surface bundle fixes the boundary of each page pointwise.
In the terminology of \cite{LVHMW}, $M$ is the {\em paper} of the open book, and  the tubular neighborhood of the binding is its {\em spine}. In the classical case, the spine is a union of solid tori $S^1 \times D^2$, where each solid torus corresponds to a boundary component of the page (a binding component), so that each page intersects $\partial (S^1 \times D^2)$ along a longitudinal curve. We will consider a generalization where the monodromy of the surface bundle is allowed to permute the boundary components of the page. Accordingly, each solid torus component of the spine may now intersect a page along $d$ longitudinal curves with $d>1$; these correspond to the boundary components permuted by the monodromy. Spinal open books in \cite{LVHMW, LVHMW2, HRW} are in fact much more general and  have another new feature: the spine may have connected components of the form $S^1 \times \Sigma$, where $\Sigma$ is an arbitrary oriented surface, possibly with multiple boundary components. {We will not need this broader generalization;
all our open books  satisfy a more restrictive assumption
that the spine is always the union of solid tori. 
This assumption implies that our open books are also a special case of rational open books of~\cite{BEVHM}. The key 
examples are multilinks of~\cite{EisNeu}, see below.
We give a definition for the specific class of the open books we need.}

\begin{definition} \label{def:spinal} We say that 
a connected closed oriented  $3$-manifold $Y$ 
has a {{\em spinal open book decomposition with multilink binding}} $\beta$ if

(1) The complement $M  = Y \setminus \nu(\beta)$ of a tubular neighborhood of the oriented link $\beta$ fibers over $S^1$.  The fibers of the fibration $\pi: M \to S^1$, called pages,  are compact oriented surfaces with nonempty boundary contained in $\d M$.  The pages meet $\d M$ transversely. 

(2) The boundary of each page is the union of parallel longitudinal copies of binding components in $\d \overline{ \nu(\beta)}$.

We say 
that the binding component $\beta_i$ has multiplicity $d_i$ if there are $d_i$ parallel copies of $\beta_i$ in the boundary of each page.

As in the classical case, a spinal open book \ls{with multilink binding} on a 3-manifold $Y$ supports the (co-oriented) 
contact structure $\xi$ with a contact form $\alpha$ if the binding is positively transverse to the contact planes, and $d \alpha$ induces an area form on each page. The contact struture supported by a given spinal open book is unique up to isotopy, \cite{LVHMW}. 
\end{definition}

{We will typically omit the words ``multilink binding'' for brevity, since all spinal open books in this paper are as in~Definition~\ref{def:spinal}.}

{A model for the neighborhood of a multilink binding component of multiplicity $m$ looks as follows. Consider $\R^3 =\{(r, \theta, z )\}$
with cylindrical coordinates and 
its standard contact structure $\ker({dz +r^2 \,d\theta})$. In the complement of the $z$-axis, the fibration over $S^1$ given by the vertical planes $\theta= const$ gives a model for a neighborhood of the binding of a classical open book. For a binding component of multiplicity $m$, the model is given by the pullback of this open book 
via the branched covering map $(z, w) \mapsto (z, w^m)$, where $w=r e^{i\theta}$. Each page meets
this spine component along $m$ vertical planes in this model, cutting $m$ 
parallel longitudes on its boundary torus. If we label the corresponding boundary components of the page by $1, \dots, m$, in the cyclic 
order given by their $\theta$-coordinate, then the monodromy of the open book permutes these page boundary components by the cyclic permutation $(12 \dots m)$. (We discuss spinal monodromy in more detail in Section~\ref{sec:diagrams}.)
When the multilink binding has several components, 
the monodromy permutation of the boundary components of the page decomposes as the product of cycles that correspond to the individual binding components. 
Up to isotopy, a spinal open book with multilink binding can be reconstructed from its page and the monodromy, by gluing the mapping torus part of the open book to  the solid tori that correspond to the binding, according to standard models.}

{An open book with multilink binding is a particular case of a spinal open book {\em uniform with respect to the disk}, 
\cite[Definition 2.3]{HRW}. This latter notion refers to more general spinal open books, where instead of the tubular neighborhood of each binding component, one has  spine components, which are $S^1$-bundles over arbitrary vertebrae surfaces $\Sigma_i$. In the multilink case, the spine components are solid tori, and each $\Sigma_i$ is the cocore disk 
of the corresponding solid torus. 
The uniform condition requires that each $\Sigma_i$ be   
a simple branched covering of $D^2$. The model of the multilink binding 
gives a cyclic covering of $D^2$ by the cocore of the multilink spine; we obtain a simple branched covering by deforming the standard model. 
If the multilink binding component has multiplicity $m$, then the simple covering will have $m-1$ simple branch points. Note that any two such branched coverings of the disk by a disk are equivalent \cite{BerEdm}, and that both the multilink and the spinal models support the same contact structure near the binding. We will need both models in this paper, as appropriate for specific statements: multilinks arise naturally when open books are induced by holomorphic functions, while the spinal structures are important in the context of (nearly) Lefshetz fibration.}

{We discuss the setting of singularities in some detail.}
If $(X, 0)$ is a normal surface singularity, a Milnor open book on its link is induced by a holomorphic function $f: X \to \C$ with $f(0)=0$, see \cite{CNPP} for a detailed discussion. To get a classical 
open book, one needs to assume  that $f$ defines the isolated singularity at 0 (this hypothesis means that the strict transform $\div_s (\pi \circ f)$ of $f$ in the embedded resolution $\pi:\ti X \to X$ of $f$ is reduced). Any such open book supports the canonical contact structure $\xi_{can}$ on the link.
If $f$ vanishes to a higher order, most of the arguments go through verbatim.
The only change is that the local model along a binding component is now given by $\arg z^n$, $n>1$, and one gets a spinal open book {with multilink binding} as above.

In more detail, fix an embedding of the germ $(X, 0)$ into $\C^N$ and the rug function $\rho(z_1, \dots z_N)= \sum_{i=1}^N |z_i|^2$. For a small $\epsilon >0$, the link is 
$Y_{\rho, \epsilon}= X \cap \{\rho=\epsilon\}$. The canonical contact structure $\xi_{can}$ on $Y_{\rho,\epsilon}$ is the kernel of $\alpha = -d^{\C}\rho$ (restricted to the tangent space of $Y_{\rho,\epsilon}$).
The function $\arg f$ gives a fibration on the complement of the binding compatible with the canonical contact structure on  $Y_{\rho, \epsilon} \setminus \{f= 0\}$: this is established in~\cite{Hamm,CNPP}, and the relevant calculations apply directly in our case. To examine 
the structure of the open book near the binding, consider an embedded 
resolution of $f$ with normal crossings, $\pi: \ti{X} \to X$,  where (the preimage of) the link $Y = Y_{\rho,\epsilon}$  is the boundary of a small tubular neighborhood of the exceptional divisor $\pi^{-1}(0)$. Then, the strict transform  $\div_s (\pi \circ f)$ is given by complex disks transverse to some exceptional spheres. The fibration given by 
$\arg f$ on $Y \setminus \{f= 0\}$ pulls back to the fibration on the complement of $\{\pi \circ f =0\}$.
 If  $\div_s (\pi \circ f)$ is reduced, then each of the transverse disks   
 is locally modelled on $\{z=0\}$, and the binding component cut out on $Y$ has the open book modeled on $\arg z$ in its neighborhood. In the non-reduced case where the given disk has multiplicity $d>1$ in 
 $\div_s (\pi \circ f)$ -- that is, $f$ vanishes to order $d$ on the corresponding branch of $\{f =0\}$, -- 
 the model is given by $\arg z^d$, with the pages of the open book approaching the binding component along $d$ longitudes.  Standard arguments show that once $f$ is fixed, 
 the open book is independent of choices made in the construction of the link, up to isomorphism. We have

\begin{theorem} [c.f.~\cite{Hamm,CNPP}] \label{MilnorOB}
There is a $C^{\infty}$ fibration 
$\arg f: Y \setminus \{f= 0\} \to S^1$ which induces a {spinal open book with multilink binding} on $Y$. 

The multiplicity of each binding component is given by the order of vanishing of $f$ on the corresponding branch of $\{f=0 \}$. If $f$ defines the isolated singularity at 0, then the pages meet each binding component along a longitude, that is, the fibration gives an open book in the classical sense. 

The spinal open book supports the canonical contact structure $\xi_{can}$ on the link.
\end{theorem}

\begin{proof} The only statement that remains to check is the compatibility of the contact structure and the open book along the binding.
We show that if a vector $V$ is positively tangent to the binding, then $\alpha(V)>0$ for the contact form $\alpha$ defined above. In an ordinary open book, the orientation on the binding is induced as the boundary orientation of a page. In our spinal context, the boundary of the page may multiply cover the binding, but it induces an orientation in the same way. Another way to see this is to consider a page as the complement of a neighborhood of the binding $\arg(f)^{-1}(\theta_0)\cap Y_{\rho,\epsilon} \cap \{|f|<\eta\}$ for sufficiently small $0<\eta<\varepsilon$. This page has genuine boundary which is a cable of the binding (consisting of parallel longitudes), with a boundary orientation induced in the usual way. As $\eta\to 0$, the cable collapses to the binding and the orientation on the cable collapses to an orientation on the binding $f^{-1}(0)\cap Y_{\rho,\epsilon}$. With respect to this orientation on $f^{-1}(0)\cap Y_{\rho,\epsilon}$, we want to see that $\alpha=-d^{\C}\rho$ evaluates positively. First we show that the orientation on $f^{-1}(0)\cap Y_{\rho,\epsilon}$ as the binding agrees with its orientation as the boundary of $f^{-1}(0)\cap \{\rho\leq \varepsilon\}$ (with its complex orientation). To see this, note that there is a deformation $f^{-1}(t)\cap (\rho\leq \varepsilon)$ from $f^{-1}(\eta)\cap (\rho\leq \varepsilon)$ to $f^{-1}(0)\cap (\rho\leq \varepsilon)$ for $0<\eta<\epsilon$ sufficiently small, which carries the orientation on $f^{-1}(\eta)\cap (\rho\leq \varepsilon)$ as the boundary of $f^{-1}(\eta)\cap (\rho\leq \varepsilon)$ to the orientation on $f^{-1}(0)\cap (\rho\leq \varepsilon)$ as the boundary of $f^{-1}(0)\cap (\rho\leq \varepsilon)$ (possibly collapsing as a multiple cover when $t=0$). Observe that there is an isotopy between $f^{-1}(\eta)$ and a page of the spinal open book $arg(f)^{-1}(1)$, by rescaling the coordinate of the radial component of $f$, and this carries the boundary orientation of $f^{-1}(\eta)\cap (\rho=\varepsilon)$ to the binding orientation. Thus, the binding orientation agrees with the orientation as the boundary of $f^{-1}(0)\cap (\rho\leq \varepsilon)$. Now since $f^{-1}(0)\setminus 0$ is a complex curve, a vector $V$ tangent to $f^{-1}(0)\cap (\rho =\varepsilon)$ induces the boundary orientation if and only if $-JV$ is an outward normal vector to $f^{-1}(0)\cap (\rho\leq \varepsilon)$, because $(-JV,V)$ is the positive complex orientation on $f^{-1}(0)$.  The vector $-JV$ is outward normal if and only if $d\rho(-JV)>0$. Thus $V$ positively orients the binding if and only if $d\rho(-JV) = -d^{\C}\rho (V) = \alpha(V)>0$, as desired.
\end{proof}

Turning back to sandwiched singularities, we now use a Milnor-type fibration to construct a spinal open book from a decorated germ. We fix an identification of the abstract link with the boundary of the singular surface (or its resolution); this fixes a plumbing structure on $Y$.  Notation in the lemma below refers to Notation~\ref{def:EC}.

\begin{lemma}\label{binding-from-function} Let $(X, 0)$ be a sandwiched singularity with the decorated germ $(C, w)$ and the contact link $(Y, \xi)$.
There exists a planar spinal open book for $(Y , \xi)$ with the following {multilink} binding data. The binding $\beta= \cup_{i=0}^m \beta_i$ is given by the union of $S^1$-fibers of the plumbing structure on $Y$, where  $\beta_i$ is a fiber of the  $S^1$-bundle over $E(C_i)$, $i>0$, and $\beta_0$ is a fiber over $E_0$.
For $i>0$, the multiplicity of the {multilink} binding component $\beta_i$ equals the multiplicity of $C_i$ at $0$, while $\beta_0$ has multiplicity 1 and corresponds to the outer boundary component of the page.
\end{lemma}

\begin{proof} Starting with the decorated germ, we will construct a related holomorphic function on $(X, 0)$ inducing the desired Milnor fibration.  
Consider  the embedded resolution of $(C, w)$  with normal crossings that corresponds to the augmented sandwiched graph. Let $\alpha: B \#_{j=1}^n \cptwobar \to B$ denote the blow-up morphism. The blow-up $\alpha^{-1}(U)$ of a neighborhood of the origin in $B$ is a neighborhood of $\alpha^{-1}(0)$ given by the plumbing of disk bundles over the exceptional curves;  
strict transforms $\ti{C_i}$ of the curvettas $C_i$ are fibers of the bundles over the $(-1)$ curves. It follows that 
$\alpha^{-1}(U) - \cup_i \nu(\ti{C_i})$ is isomorphic to the  resolution $\ti{X}$. Let $f:\C^2 \to \C$ be the projection $f(x, y) = x$. Its fibers are vertical planes. Its pullback $\alpha \circ f: B \#_{i=1}^n \cptwobar \to B$ restricts to a holomorphic function on  $\ti{X}$ via the isomorphism above. If we contract the exceptional curves,  $\alpha \circ f$ restricts to a continuous function $g$ on $(X, 0)$, holomorphic away from $0$. By normality, $g$ is then holomorphic on $(X, 0)$. The fibration 
$$
\arg g: X \setminus \{ g=0 \} \to S^1 
$$
gives a (spinal) Milnor open book with planar fibers on $Y = \d X$ whose binding is $\{g=0\} \cap Y$. To describe the binding in terms of the plumbing, consider the total transform $\bar{L}$ in  $B \#_{i=1}^n\cptwobar$ of the curve $L=\{x=0\} \subset \C^2$; we have $\bar{L} = \div (\alpha \circ f)$. The multiplicity $d_i$ of the curvetta $C_i$ at $0$ is given by the intersection with a generic line, 
so we have 
$$
L \cdot C_i = d_i = \text{ degree of the branched covering } C_i \to D_x.
$$
Here and below, $L\cdot C_i$ denotes the intersection number of the algebraic curves which must be defined carefully since the curves have boundary. One could interpret $L \cdot C_i$ as the intersection of chains with fixed disjoint boundaries (which can be computed by perturbing the chains away from the boundary to make them transverse). For a more concrete definition of the intersection number of two algebraic curves, we can
define this quantity at a single intersection point 
 by using the algebraic equations of the two curves, 
plugging in a primitive parametrization for one curve $(x,y)=(\phi(t),\psi(t))$ into the defining equation $g(x,y)$ for the other and looking at the degree of $g(\phi(t),\psi(t))$. Summing over all points of intersection yields a well-defined intersection number, see~\cite[Chapter 1.2]{Wall}. Under blow-ups, 
this intersection number changes as follows: for any curves $A$ and $B$ intersecting at a point $p$, $\bar{A}\cdot \ti{B} = A\cdot B$ where $\bar{A}$ is the total transform of $A$ under a blow-up at $p$ and $\ti{B}$ is the strict transform of $B$.  (To see this is true, suppose $A$ has multiplicity $k$ at $p$ and $B$ has multiplicity $\ell$ at $p$. Then the exceptional divisor $E$ appears with multiplicity $k$ in $\bar{A}$, so $\bar{A} = \ti{A}\cup kE$. Thus 
	$$\bar{A}\cdot \tilde{B} = \tilde{A}\cdot \tilde{B} + k E\cdot \tilde{B}. $$
By~\cite[Lemma 4.4.1]{Wall}, $\ti{A}\cdot \ti{B} = A\cdot B - k\ell$. By~\cite[Lemma 3.4.7]{Wall}, $E\cdot \ti{B} = \ell$. Thus $\bar{A}\cdot \ti{B} = A\cdot B$ as claimed.)

Using this fact for $L$, $C_i$, and the  total transform $\bar{L}$ and strict transform $\ti{C}_i$ after each blow-up in the sequence, we see that after all the blow-ups leading to the augmented resolution, $\bar{L}\cdot \ti{C}_i=L\cdot C_i = d_i$.

%Since the total transform $\overline{L}$ is homologous to $L$, and the compact part $\overline{C}_i - \ti{C}_i$ can be made disjoint from $L$, we have 
%$$
%d_i = L \cdot C_i = L \cdot \overline{C}_i = L \cdot \ti{C}_i  = \overline{L} \cdot  \ti{C}_i.
%$$

Since $\ti{C}_i$ is a smooth disk transverse to the corresponding $(-1)$ curve attached to $E(C_i)$ in  $B \#_{i=1}^n \cptwobar$ and otherwise disjoint from the components of $\bar{L}$, it follows that this $(-1)$ curve is contained in $L' =\div (\alpha \circ f)$ with multiplicity $d_i$. In $\alpha^{-1}(U) - \cup_i \nu(C_i)$, what's left of this $(-1)$ curve is a fiber over $E(C_i)$ in the plumbing structure.  Therefore, in $\ti{X}$ the vanishing locus of the restriction of $\alpha \circ f$ is given by the union of fibers over the curves $E(C_i)$, with corresponding multiplicities, together with the strict transform of $\{x=0\}$, which is the fiber over $E_0$ with multiplicity 1. In turn, this gives the binding data  for the spinal Milnor open book at $Y= \d (\alpha^{-1}(U) - \cup_i \nu(C_i))$ induced by the function $g$. Since every Milnor open book supports the canonical contact structure
by Theorem~\ref{MilnorOB}, this proves the lemma.   
\end{proof}

\subsection{Nearly Lefschetz fibrations}

{To describe symplectic fillings for spinal open books in our setting, we rely on the results of~\cite{HRW}. 
Before formal definitions,  Example~\ref{ex:squareroot} will give a prototype and a local model for the spinal planar open books and nearly Lefschetz fibrations we work with.}

\begin{example} \label{ex:squareroot} Consider  the curve $C=\{y^2 = x\} \subset \C^2$, 
which is a branched double cover of $\C= \C_{x}$ under the projection
$\pi_x: \C^2 \to \C_{x}$ to the first coordinate, with the branch point at $0$. Let $D_x\subset \C_{x}$ and $D_y \subset \C_y$ be the disks centered at $0$ in the coordinate planes, and let $W = (D_x \times D_y) \setminus \nu(C)$ be the complement of a small tubular neighborhood of $C$ in $D_x \times D_y $. This is a manifold with corners; we will always assume that the corners are smoothed. The 3-manifold $\partial W$ has a spinal open book decomposition: $\partial W = \partial_{vert} W \cup \partial_{hor} W$, 
where the vertical boundary  
$\partial_{vert} W = \pi_x^{-1}(\partial D_x)$ is the surface bundle over $S^1 = \partial D_x$ whose fiber is a disk with two holes. 
This surface bundle is the ``paper'' of the spinal open book. We write $P_\theta$ for the fiber over $\theta \in \d D_x$. The ``spine'' is given by the horizontal boundary 
$\partial_{hor} W = T_{outer} \cup T_{inner}$, which is the disjoint union of two solid tori.  The first solid torus 
$T_{outer}= D_x \times \partial D_y$ corresponds to the binding component given by the outer boundary of the page, and is glued to the boundary of the ``paper'' surface bundle as usual (for each $\theta\in \d D_x$, the outer boundary of the page $P_\theta$ gives one longitudinal curve $\{\theta\}\times S^1$ in $\d T_{outer}$). The second solid torus $T_{inner}=D^2 \times S^1$ corresponds to {\em both} remaining boundary components of the page, each giving a longitudinal curve on $\d T_{inner}=\d D^2 \times S^1$: the longitudes $\{\theta\} \times S^1$ and
$\{\theta +\pi \} \times S^1$ of $\d T_{inner}$ are identified with 
the boundaries of the two holes in $P_{2 \theta}$.
%where we think of $\theta$\comm{$2\theta$?} as the angular coordinate in $\partial D_x$. 
The projection of a meridian 
$\d D^2 \times \{s\}$ of $T_{inner}$ to $\d D_x$ is a double cover.

The monodromy of this spinal open book is a {\em boundary interchange} of the two inner boundary components: it is a positive half-twist fixing the parameterization of each component as in  Figure~\ref{fig:squareroot}. 
\end{example}

\begin{figure}[h!]
\centering
\includegraphics[scale=0.5]{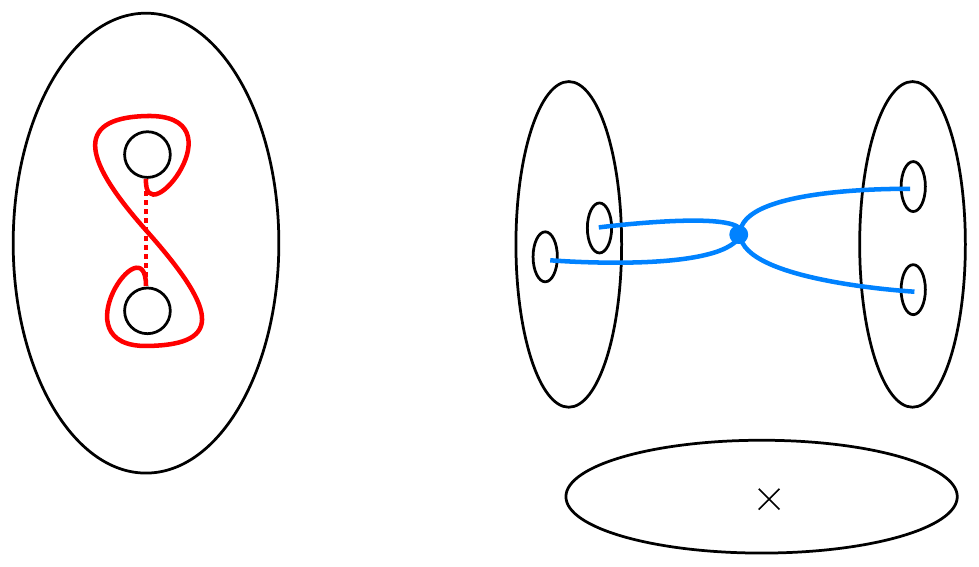}
\caption{The boundary interchange is the monodromy of the complement of the tubular neighborhood of the curve $\{y^2 = x\}$.}
\label{fig:squareroot}
\end{figure}

Nearly Lefschetz fibrations, studied in \cite{HRW}, are a generalization of the classical notion, with exotic fibers in addition to Lefschetz singularities. 

\begin{definition} Let $W$ be a $4$-manifold with corners, 
with boundary decomposed as $\d W = \d_{hor} W \cup \d_{vert} W$ so that  $\d_{hor} W$ and  $\d_{vert} W$ intersect over a collection of tori.  
A (bordered) nearly Lefschetz fibration over the disk   $D^2$ is 
a projection $\pi: W\to D^2$ such that  

\noindent (1) the projection $\pi$ gives a smooth fiber bundle over $\d D^2$, so that $\pi^{-1}(\d D^2) = \d_{vert} W$;

\noindent (2) $\pi$ has finitely many critical points,  each modelled on the map $(z_1, z_2) \mapsto z_1^2 + z_2^2$, with distinct critical values;

\noindent (3) there are finitely many exotic fibers, each locally modeled on  Example~\ref{ex:squareroot};

\noindent (4) away from the exotic fibers, the restriction of $\pi$ to $\d_{hor} W$ gives a smooth fiber bundle.

\noindent The fibers of the projection are assumed to be connected, and $\d_{hor} W$ is the union of the boundaries of the fibers.

\end{definition}

We refer the reader to~\cite[Section 2]{HRW} for a detailed discussion of nearly Lefschetz fibrations. All nearly Lefschetz fibrations will be positive and allowable (PANLF in the terminology of \cite{HRW}), and we will usually omit these words for brevity. The ``allowable'' hypothesis means that the vanishing cycles are homologically essential curves in the fiber. 
 
{Nearly Lefschetz fibrations naturally arise as the complements of positive multisections of Lefschetz fibrations, \cite{BaHa}. A positive multisection is a simple branched cover of the base, with the branch points modeled on the projection to $\C_x$ of the curve $y^2 = x$ in $\C^2=\C_x \times \C_y$.

A generic projection of a smooth complex curve in $\C^2$ gives a multisection of the product fibration $\C^2 \to \C$. Baykur and Hayano \cite{BaHa} show that given a positive multisection in a positive allowable Lefschetz fibration, the complement of the multisection has the structure of positive allowable nearly Lefschetz fibration. The fiber of this new fibration is the fiber of the original fibration with holes. 
The branch points of the multisection correspond to the exotic fibers where two holes are merged together.
The converse is also true: by patching together local models, one can show that every nearly Lefschetz fibration is in fact the complement of a multisection in some honest Lefschetz fibration.  See \cite[Theorem 1.1]{BaHa} and also \cite[Theorem 2.13]{HRW} for a statement that is geared more closely to our case.}

{A spinal open book structure is induced on the boundary of a 4-manifold equipped with a  nearly Lefschetz fibration.  The mapping torus part of the spinal open book is given by the vertical boundary of the fibration, and the spine is given by its horizontal boundary. This setting shows how the spine may have complicated topology, in general:   when the nearly Lefschetz fibration is given as the complement of a multisection, the connected components of the spine are $S^1$-bundles over the surfaces that are the connected components of the multisection.
(These surfaces themselves are {\em vertebrae} in the terminology of \cite{LVHMW, HRW}.) In general, the components of the spine are not necessarily solid tori. If we want the boundary to have the structure of a spinal open book with multilink binding, as in Definition~\ref{def:spinal}, then we must require that 
every  multisection component be a disk. If the multisection component is a disk with a $d$-fold branched covering of the base, then the corresponding spine component is a solid torus, and each page of the spinal open book cuts $d$ longitudinal curves on its boundary. This can be modeled on a multilink binding component 
as in Definition~\ref{def:spinal}, after deforming the single $d$-fold branch point on the cross-section of the solid torus into $(d-1)$ simple branch points. In what follows, our spinal open books will all be planar,
with a planar page $P$ identified with a disk with holes. The outer boundary component of $P$ 
corresponds to a binding component $\beta_0$ of multipicity 1: 
each parallel longitude is identified with the outer boundary circle of a single page. The holes of the page 
may correspond to binding components of higher multiplicity.}

 %The {\em spinal mapping class group} $\SMod(P)$ of a surface $P$ with boundary is, by definition, generated by Dehn twists and boundary interchanges, where  a {\em boundary interchange} 
 %of two boundary components along an arc in the page is defined as in the model in Figure~\ref{fig:squareroot}.

% \begin{remark}  In our setting, the 
%monodromy $\phi \in SMod(P)$ will always fix the outer boundary component of $P$. The paper $M$ of this open book is the surface bundle over $S^1$ with fiber $P$ and monodromy $\phi$, and the spine is the union of solid tori 
%$T_i$, $i=0, \dots k$. The solid torus
%$T_0$ corresponds to the outer component of $\d P$ fixed by $\phi$,
%so this spine component is glued to paper
%as in the classical case: each longitude is identified with the outer boundary circle of a single page. The corresponding binding component $\beta_0$ has multiplicity $d_0=1$.  \ls{For $i>0$,  each spine component $T_i$ corresponds to the orbit of a component of $\d P$ under permutations by iterations of monodromy $\phi$. The length of this orbit is the multiplicity $d_i$ of the corresponding binding component $\beta_i$:   each of the boundary 
%circles in this orbit  is identified with a longitudinal curve in $\d T_i$, so each page meets $\d T_i$ along $d_i$ longitudes. The projection $M \to S^1$ 
 % restricts to the meridional curve of  $T_i$ as a $d_i$-fold covering of  $S^1$; the meridional curve of  $T_0$ is mapped homeomorphically to $S^1$.} \comm{Mention the spine component is still a solid torus?}
%\end{remark}

We will use nearly Lefschetz fibrations on the Milnor fibers of smoothings arising from the de Jong--van Straten construction and on Stein fillings arising similarly from the DJVS arrangements.   As before, we work in a Milnor ball $B= D_x \times D_y$, with the corners smoothed. A DJVS arrangement (or a picture deformation) allows for intersections between the curves at transverse multiple points, so it is not a multisection of the product fibration in $B$.
Blowing up at the marked points $p_j$, we  consider the Lefschetz fibration 
$(D_x \times D_y)  \#_n \cptwobar \to D_x$, where each singular fiber contains a copy of the exceptional sphere. 
(The regular fiber is a disk, and the vanishing cycles are nullhomotopic.)
The strict transforms $\tilde{\Gamma}_i$ are disjoint, so the arrangement now becomes a multisection. 

\begin{theorem} \label{thm:smooth-nearlyLefschetz} Suppose that 
$(\Gamma, p)$ is a DJVS immersed disk arrangement.  
Let the manifold 
$$
W_{(\Gamma, p)} = [(D_x \times D_y)  \#_n \cptwobar] \setminus \cup_i \nu(\tilde{\Gamma}_i),
$$
be the complement of the union of small tubular neighborhoods of the strict transforms $\tilde{\Gamma}_i$ of the  disks~$\Gamma_i$ in the blow-up of $D_x \times D_y$ at all the marked points $p_j \in p$.   
Then $W_{(\Gamma, p)}$ carries a positive allowable nearly Lefschetz
fibration with an exotic fiber for each branch point and one singular fiber for each marked point $p_j$. The vanishing cycle $V_j$ corresponding to the marked 
point $p_j$ encloses the holes in the planar fiber that correspond to the disks $\Gamma_i$ meeting at $p_j$.    In particular, a free marked point on $\Gamma_i$ gives a boundary-parallel vanishing cycle around one of the holes corresponding to the intersection of $\Gamma_i$ with $D_y$.   
%\comm{Should we say something about the spine in this theorem?}
\end{theorem}

\begin{proof} Using the local model calculation of \cite[Lemma 3.2]{PS} in a neighborhood of an intersection point of 
$\Gamma_i$'s, 
we see that the blow-up at the marked point $p_j$ produces a Lefschetz fibration whose singular fiber contains the exceptional sphere. The map is the composition of the blow-down with the projection to $D_x$. The strict transforms of the disks $\Gamma_i$ meeting at the point $p_j$ pass through distinct points of that sphere since the interesection is transverse. The corresponding vanishing cycle in the nearby fiber of this
Lefschetz fibration encloses the intersections of these disks $\Gamma_i$ with the fiber. Thus, after deleting the neighborhood of the strict transforms of the $\Gamma_i$, these vanishing cycles become essential curves in the fiber, so the Lefschetz singularities will be allowable. The strict transforms 
$\tilde{\Gamma}_i$ form multisections of the Lefschetz fibration on $(D_x \times D_y)  \#_n \cptwobar$, so by the arguments of~\cite{HRW, BaHa} the local model for the complement of (the tubular neighborhoods of) $\tilde{\Gamma}_i$ near a branch point agrees with the local model of an exotic fiber. 
%(Note that Baykur--Hayano's theorem \cite{BaHa} does not apply as stated because the original fibration on the blow-up of the disk is not allowable, but in our setting the argument goes through.)
%%%% actually Baykur-Hayano have the statement in full generality, incl for non-minimal fibrations %%%%%%%
Thus, $W_{(\Gamma, p)}$ carries a positive allowable nearly Lefschetz fibration with the exotic fibers corresponding to the branch points.
\end{proof}

\subsection{Compatibility with the canonical contact structure} Lefschetz fibrations as above provide fillings of the link of the singularity with the given decorated germ, but we need to establish compatibility with the contact structure.

\begin{theorem} \label{OB-compatible} Let $(Y, \xi)$ be the contact link of a sandwiched singularity $(X, 0)$ with a decorated germ $(C, w)$. Let $C^t = \cup_i C^t_i$ be a picture deformation of $(C, w)$, and $W$ the Milnor fiber obtained by the DJVS construction.  

Then the contact structure $\xi$ is supported 
by the spinal open book induced on the boundary $\d W=Y$ by the nearly Lefschetz fibration of Theorem~\ref{thm:smooth-nearlyLefschetz} constructed from the picture deformation $C^t$. 
The same is true for any nearly Lefschetz fibration $W_{(\Gamma, p)}$ corresponding to a DJVS arrangement $(\Gamma,p)$,
provided that $(\Gamma,p)$ is compatible with $(C,w)$.

%the boundary link $\cup_i \d \Gamma_i = \Gamma_i\cap \partial(D_x \times D_y)$ of 
%the arrangement $\Gamma_i$ is isotopic 
%to $\cup_i \d C_i$
%in
%$\partial D_x \times D_y$ and $w(\Gamma_i)= w_i$ for all $i=1, \dots, m$.
\end{theorem}

To prove Theorem~\ref{OB-compatible}, we identify the binding components of the spinal open book induced by the Lefschetz fibration as a union of fibers of the plumbing structure of the link $Y$. Theorem~\ref{OB-compatible}  follows from Lemma~\ref{binding-from-fibration}  and Lemma~\ref{binding-from-function},  together with the fact that the isotopy class of a spinal open book {with multilink binding} on a rational homology sphere is uniquely determined by its {multilink} binding (Lemma~\ref{OBisotopy}).

\begin{lemma}\label{binding-from-fibration} Fix a decorated germ $C=\cup_i C_i$ encoding $(X, 0)$, let $(\Gamma_i, p)$ be a DJVS disk arrangement compatible with $(C, w)$. Consider a nearly Lefschetz fibration on 
$W=W_{(\Gamma, p)}$ as in Theorem~\ref{thm:smooth-nearlyLefschetz}. Then the binding $\beta$ of the induced spinal open book on $Y = \partial W$ is given by the union of $S^1$-fibers of the plumbing structure on $Y$, so that  
$
\beta= \cup_{i=0}^m \beta_i$,  where  $\beta_i$ is a fiber of the  $S^1$-bundle over $E(C_i)$, $i>0$, and $\beta_0$ is a fiber over $E_0$.
For $i>0$, the degree of the {multilink} binding component $\beta_i$ equals the multiplicity of $C_i$ at $0$, while $\beta_0$ has degree 1 and corresponds to the outer boundary component of the page. 
\end{lemma}

\begin{proof}

To relate the boundary of the Milnor fibers and fillings to the plumbing structure, we will cap them off by a cap obtained from the resolution. Consider the cap $U$ as in Section~\ref{sandwich-setup}, built by gluing neighborhoods of the strict transforms of the $\Gamma_i$ as $2$-handles to a ball $({B}', {{C}}')$ attached along the boundary of $C_i'$. The union of the proper transform of $\Gamma_i$ with $C_i$ is a sphere $\tilde{\Sigma}_i$ which has at most one singularity at the origin in $B'$ corresponding to a potential singularity of $C_i'$.
%for every $W=W_{(\Gamma, p)}$ as above,  $W \cup U$ is a blow-up of the 4-sphere. Accordingly, we form the union of $(B, \Gamma)$ and $({B}', {\mathcal{C}}')$ so that the boundary of $\Gamma_i$ is glued to the boundary of the corresponding germ branch ${C}'_i$. Then we blow up at all points $p_j\in p \subset B$ to consider surfaces $\tilde{\Sigma}_i$ that are the strict transforms of the immersed singular surfaces $\Sigma_i=\Gamma_i \cup {C}_i'$. Each $\tilde{\Sigma}_i$ is smoothly embedded in $\#_{i=1}^n \cptwobar$ away from the origin in $B'$, where it has a singular point if ${C}_i'$ has one. 
Next, we would like to blow up at the origin in $B'$ and infinitely near points to find an embedded resolution of the germ $C'$ encoding the dual resolution graph of the original singularity $(X, 0)$ with the added $(-1)$ vertices, as at the start of Section~\ref{sandwich-setup}.  \op(This part of the argument is similar to \cite[Lemma 4.2.5]{NPP}.)
 Because of the orientation reversal in $(B', C')$, we have to reverse the orientations throughout and perform anti-blow-ups. We then get  $\#_{i=1}^l \cptwo \#_{i=1}^n \cptwobar$ with smoothly embedded spheres 
$\overset{\scriptscriptstyle\approx}{\Sigma}_i$ formed by the strict transforms of $\Gamma_i$ and $C_i$ (with the orientation reversal on the part corresponding to $B'$). We claim that 
$\overset{\scriptscriptstyle\approx}{\Sigma}_i \cdot \overset{\scriptscriptstyle\approx}{\Sigma}_i =0$. Indeed, in the complex setting the self-intersection of a curve drops by $m^2$ when blowing up a point of multiplicity $m$. This means that 
when we blow up the germ $C$ to produce the required embedded resolution, the self-intersection of the relative cycle corresponding to $C_i$ drops by $\sum_j w^2_{i,j}$, where $w_{i,j}$'s are the multiplicities of the infinitely near points on the strict transforms as in Section~\ref{sandwich-setup}. For each $i$, $\delta_i=\delta(C_i)$ can be expressed via the (non-zero) quantities $w_{i,j}$  
as 
$$
\delta_i = \sum_j \frac{w_{i,j}(w_{i,j}-1)}2. 
$$
Since $w_i = \sum_ j w_{i,j}$, it follows that the self-intersection of $C_i$ under complex blow-ups will drop by $w_i + 2 \delta_i$. Under the orientation reversal and anti-blow-ups, self-intersections {\em increase} by the corresponding quantities; by~\eqref{eq:Sigma-self-int}, it 
follows that 
$$
\overset{\scriptscriptstyle\approx}{\Sigma}_i \cdot \overset{\scriptscriptstyle\approx}{\Sigma}_i = {\ti{\Sigma}}_i \cdot {\ti{\Sigma}}_i + w_i + 2 \delta_i =0.
$$

Now we can describe the cap directly as a plumbing of disk bundles over spheres: it is given by the plumbing corresponding to the resolution with the additional disk bundles of Euler number $(-1)$ over spheres, as dictated by the augmented resolution graph, and the disk bundles of Euler number 0 over spheres $\overset{\scriptscriptstyle\approx}{\Sigma}_i$. The disk bundle over $\overset{\scriptscriptstyle\approx}{\Sigma}_i$ is plumbed on the $(-1)$ disk bundle that corresponds to the $(-1)$ vertex of the augmented graph adjacent to the curvetta $\ti{C}_i$.  See Figure~\ref{fig:capkirby}.
In the 3-manifold, the surgery induced by the $0$-framed handles cancels that of the $(-1)$-framed handles. Thus, we obtain an identification of the boundary of the cap (and thus the boundary of the Milnor fibers and fillings) with the boundary of the original plumbing from the resolution. 

\begin{figure}
	\centering
	\includegraphics[scale=.5]{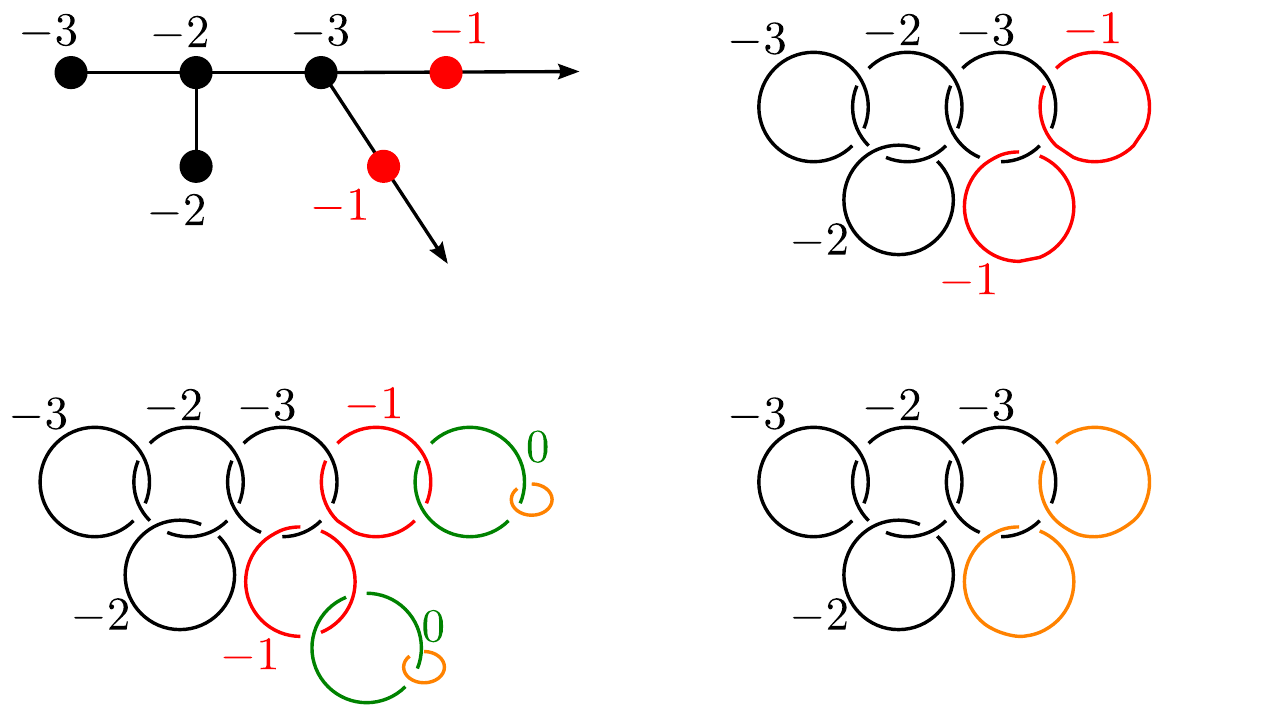}
	\caption{\textbf{Top left:} an augmentation of the plumbing graph for a normal crossing resolution of a sandwiched singularity. \textbf{Top right:} its corresponding Kirby diagram. \textbf{Bottom left:} the cap with additional disk bundles of Euler number $0$ over spheres $\ti{\ti{\Sigma}}_i$ together with meridian curves of $\ti{\ti{\Sigma}}_i$ corresponding to the markings.\\ \textbf{Bottom right:} the result of canceling the $0$-surgery with the $(-1)$ surgery, tracking the images of the markings.}
	\label{fig:capkirby}
\end{figure}

To finish the proof of the lemma, we examine the binding of the spinal open book on the boundary of the nearly Lefschetz fibration on $W$. The page of the open book is a disk with holes, where the holes correspond to the curvettas. Suppose the curvetta $C_i$ has multiplicity $d_i$, then a regular fiber of the fibration has $d_i$ boundary components corresponding to $C_i$. These $d_i$ circles all come together at one component of the spine: the spine component in $Y= \d W$ is a solid torus whose core disk is a branched cover of degree $d_i$  of the base $D_x$ of the nearly Lefschetz fibration. Since we start with the Lefschetz fibration $\alpha: B  \#_{i=1}^n \cptwobar \to D_x$ and remove the tubular neighborhoods $\nu(\Gamma_i)$ to obtain $W$, each of the $d_i$ copies of the binding arises as the boundary of a meridional disk of $\nu(\Gamma_i)$, or, viewed on the side of the cap $U$, the boundary of a co-core disk of the $0$-framed handle that corresponds to $\ti{\Sigma}_i$. In the surgery diagram for $Y= \d U$, the binding component $\beta_i$ is then the meridian of the $0$-framed circle as in Figure~\ref{fig:capkirby}. Sliding it over the $(-1)$-framed handle and canceling the $0$-handle with the $(-1)$-handle, we see that $\beta_i$ is isotopic to a fiber in the $S^1$-bundle over $E(C_i)$, as required. This takes care of the binding components that correspond to the boundaries of the holes of the disk page. The outer boundary component $\beta_0$ is the boundary of the disk page of the open book on  $B  \#_{i=1}^n \cptwobar$, before the curvetta removal. Then, $\beta_0$ is identified with the intersection of $\partial W$ with the locus 
$\{x=0\}$. In the blow-up process, the strict transform of $\{x=0\}$ intersects the first exceptional divisor $E_0$ transversely and is disjoint from the strict transforms of the $C_i$, which shows that $\beta_0$ is a fiber over $E_0$.  
\end{proof}

	\begin{lemma} \label{OBisotopy} Suppose that $Y$ is a closed oriented $3$-manifold with $b_1(Y)=0$. A 
	{spinal open book with multilink binding} on $Y$ is determined, up to isotopy,  by the isotopy class of its binding (as an oriented link) together with the multiplicities of the binding components. 
\end{lemma}

\begin{proof}	The case of classical open books was considered in~\cite{CPP}, and their proof easily adapts to our case. A (spinal) open book with spine $\beta$ gives a fibration on the manifold with boundary $Y \setminus \nu(\beta)$; the fibration is transverse to the boundary. If these fibrations are isotopic on  $Y \setminus \nu(\beta)$  for two given open books on $Y$ with binding $\beta$, then the open books are isotopic. By the classical results on foliations and fibrations due to Stallings, Waldhausen, and Thurston 
(see for example \cite[Theorem 3.1]{CantwellConlon} for an exposition based on~\cite{Thurston} of the case we need), a fibration with fiber $F$ on $M = Y \setminus \nu(\beta)$ is determined, up to isotopy, by the homology class 
$[F, \partial F]$ in $H_2(M, \partial M; \R)$. But $H_2(M, \d M; \R) \cong H_2(Y, \beta; \R)$ by excision, and the isomorphism from the long exact sequence of a pair
$H_2(Y, \beta; \R)\cong H_1(\beta; \R)$ for the rational homology sphere $Y$ tells us that a class of a surface in $H_2(M, \d M; \R)$ is determined by  the class of its boundary, which in turn is given by the class of the binding taken with multiplicities of its components. 
\end{proof}

\subsection{Nearly Lefschetz fibrations give Stein fillings}

The results of \cite[Section 5]{HRW} give a Stein structure on a manifold equipped with a nearly Lefschetz fibration as above, producing a Stein filling for the spinal open book on the boundary.

\begin{theorem} \label{thm:Steinfill} Under the hypotheses of Theorem~\ref{thm:smooth-nearlyLefschetz}, the nearly Lefschetz fibration 
$$
W_{(\Gamma, p)}= [(D_x \times D_y)  \#_n \cptwobar] \setminus \cup_i \nu(\tilde{\Gamma}_i)
$$
can be given a Stein structure that makes $W_{(\Gamma, p)}$ a Stein filling of the contact structure supported by the spinal open book on $\d W_{(\Gamma, p)}$.  
\end{theorem}

Theorem~\ref{thm:Steinfill} follows immediately from~\cite[Theorem E]{HRW}, which applies because  the monodromy fixes the outer boundary component of the planar page and thus the corresponding spine component has multiplicity 1 and no branch points. The Stein structure in the complement of multisections is constructed in~\cite[Theorem E]{HRW} as follows: after an isotopy, one can assume that the multisection is represented by a complex curve over a small disk containing the projections of all exotic fibers but no Lefschetz singularities. The complement of a complex curve is a Stein domain by classical results, and one checks that the contact structure induced on the boundary by a plurisubharmonic function is supported by the spinal open book. Lastly, the Stein structure is extended over the Weinstein handles corresponding to the vanishing cycles. 

The next corollary is Part (1) of Theorem~\ref{thm:maintechnical}:

\begin{cor}  If the DJVS arrangement $(\Gamma, p)$ is compactible with  the decorated germ 
$(C, w)$,  then Theorem~\ref{thm:Steinfill} gives a Stein filling of the contact link $(Y, \xi_{can})$ of the link of the sandwiched singularity with the decorated germ $(C, w)$.      
\end{cor}

\subsection{All symplectic fillings arise from DJVS arrangements} 
 
{To prove Part (2) of Theorem~\ref{thm:maintechnical}, we will use the key structural results of~\cite{HRW}: {\em all} 
strong symplectic fillings of a planar spinal contact 3-manifold arise as 
compatible nearly Lefschetz fibrations inducing the fixed planar spinal open book on their boundary.} {Building on~\cite{LVHMW, BaHa2}, they prove the following.}

{\begin{theorem}[\cite{HRW} Theorem A] \label{thm:MRWmultisectioncomp}
	Let $(M,\xi)$ be a contact 3-manifold supported by a planar spinal open book and $(W,\omega)$ a minimal strong symplectic filling of $(M,\xi)$. Then $(W,\omega)$ is symplectic deformation equivalent to the complement of a neighborhood of a positive multisection in a Lefschetz fibration.
\end{theorem}}

{In our case, the spinal open book has multilink binding. The page $P$ of the open book is a disk with holes. Under this set-up, we have a more concrete topological description.} 

{\begin{theorem}\label{thm:nearlyL-sphere}
Let $(Y, \xi)$ be a contact 3-manifold supported by a planar spinal 
open book with multilink binding, such that there is at least one binding component of multiplicity one. 
Then each minimal strong symplectic filling arises as the complement of a tubular neighborhood multisection in a Lefschetz fibration
$(D^2 \times D^2) \# n \cptwobar  \to D^2$ for some $n\geq 0$. Connected components of 
the multisection are in bijective correspondence with the multilink binding components, with matching degrees/multiplicities.  Each connected component of the multisection is a disk.
\end{theorem}
}

\begin{proof}
{Let $(W,\omega)$ be a minimal strong symplectic filling of $(Y,\xi)$. By Theorem~\ref{thm:MRWmultisectioncomp}, $(W,\omega)$ is symplectic deformation equivalent to the complement of a neighborhood of a multisection in a Lefschetz fibration, such that the induced boundary is our given planar spinal open book with multilink binding. In particular, the Lefschetz fibration is over a disk. The components of the multisection are in bijective correspondence with the vertebrae of the spinal open book, and thus are disks because of the multilink condition. 
	
The fiber of the Lefschetz fibration is obtained from the fiber of the nearly Lefschetz fibration by filling in the boundary components which correspond to multisection components. In our setting, one of the binding components has multiplicity one, so we will not fill the corresponding boundary component of the page. {We can assume every other boundary component is filled, by interpreting the horizontal boundary as a disjoint union of $S^1$ bundles over surfaces $\Sigma_i$ with boundary. Since such $S^1$ bundles are trivial, we can fill them with $D^2\times \Sigma_i$, a neighborhood of a multisection component $\{0\}\times \Sigma_i$.} After filling in these multisections, the fiber of the Lefschetz fibration is a disk, so all the vanishing cycles are contractible. Thus the Lefschetz fibration must be on $(D^2\times D^2)\# n \cptwobar$ where each connected summand of $\cptwobar$ corresponds to a contractible vanishing cycle. }		
\end{proof}

\begin{proof}[Proof of Theorem~\ref{thm:maintechnical}, Part (2)]
{Given a minimal symplectic filling $W$ of the contact link $(Y, \xi_{can})$ of the link of the sandwiched singularity, the spinal open book of Lemma~\ref{binding-from-function} satisfies the hypotheses of Theorem~\ref{thm:nearlyL-sphere}. Therefore, there is a multisection 
$\tilde{\Gamma}= \cup_{i=1} \tilde{\Gamma}_i$ in a blowup of $D^2 \times D^2$, with connected components $\tilde{\Gamma}_i$, such that  
$$
W= [(D^2 \times D^2)  \#_n \cptwobar] \setminus \cup_i \nu(\tilde{\Gamma}_i).
$$
We can assume that the multisection is transverse to the fibers and intersects the boundary of $(D^2 \times D^2)  \#_n \cptwobar$ transversely. 
The vanishing cycles of the (non-minimal)  Lefschetz fibration on $(D^2 \times D^2)  \#_n \cptwobar$ are simple closed curves in the disk, and we assume that they are contained in distinct fibers. Contracting the disks bounded by each of these curves in its fiber is equivalent to blowing down the total space. Let $\Gamma_i$ denote the image of $\tilde{\Gamma}_i$ in $D^2 \times D^2$ under this blowdown. To make $\Gamma = \cup \Gamma_i$ into a DVJS arrangement, we equip it with a collection of marked points
$p=\{p_j\}_{j=1}^n$ corresponding to the vanishing cycles. A vanishing cycle that separates several multisection branches from the outer boundary of the disk  fiber gives a marked intersection point of the arrangement, and each vanishing cycle encircling a single branch gives a free marked point on that branch. Note that by minimality of $W$, a vanishing cycle cannot bound a disk in the fiber that is disjoint from the multisection $\tilde{\Gamma}$.}

{By construction, $(\Gamma, p)$ is a DJVS immersed disk arrangement, with each map $n_i$ 
given by the composition  of the smooth embedding of the 
component $\tilde{\Gamma}_i$ into $(D^2 \times D^2) \# n \cptwobar$, followed by the blowdown.  Indeed, by Theorem~\ref{thm:nearlyL-sphere}, each $\tilde{\Gamma}_i$ is a disk, so $\Gamma_i$ is an immersed disk. Intersections are positive transverse multipoints because they are created by blowing down an exceptional curve that several branches of the multisection intersect transversely.}  

{It is clear that the original filling $W$ can be recovered from $(\Gamma, p)$ by the DJVS construction, that is, $W= W_{(\Gamma, p)}$. Since the original spinal open book on the boundary of $W$ is the open book obtained from the decorated germ $(C, w)$, it also follows that  $(\Gamma, p)$ is compatible with the germ $(C, w)$.}  
\end{proof}

\begin{remark}
	Our theorem means that given a specific symplectic filling $(M,\omega)$ of $(Y,\xi_{can})$, there is a diffeomorphism between $M$ and  $W_{(\Gamma, p)}$
	given by some DJVS arrangement, 
	such that pulling back the symplectic form on $M$ to $W$ gives a symplectic form supported by the nearly Lefschetz fibration. 
	The nearly Lefschetz fibration determines a symplectic structure up to symplectic deformation equivalence~\cite[Section 4.2]{HRW}, so any symplectic form on $W$ compatible with the nearly Lefschetz fibration will be symplectic deformation equivalent to the pullback of $\omega$.
\end{remark}

\section{Monodromy factorizations and braided wiring diagrams}\label{sec:diagrams}

{For constructions and classification questions, it is important to have hands-on tools that translate some of the topological information into combinatorial data. 
For (nearly) Lefschetz fibrations, this role is played by monodromy factorizations. 
To encode DVJS arrangements, we generalize the notion of braided wiring diagram (first used by Arvola~\cite{Arvola} to study complex line arrangements).  These are $1$-dimensional singular braids, in our case decorated with marked points. 
For symplectic fillings of the links of sandwiched singularities, given by nearly Lefschetz fibrations as in the previous section, we describe the explicit relation between the monodromy factorization of the fibration and the braided wiring diagram with tangencies that encodes the corresponding DJVS arrangements. This extends our work in~\cite{PS}, where we used the language of factorizations and braided wiring diagrams to prove a version of Theorem~\ref{thm1-intro} in a more restictive setting and construct unexpected fillings from certain wiring diagrams.  In our subsequent paper~\cite{BPS}, we develop moves in a diagrammatic calculus for braided wiring diagrams with tangencies, as described in this section, and use it to construct unexpected rational homology disk Stein fillings for certain links of sandwiched singularities. 
Most of the material of this section applies to more general nearly Lefschetz fibrations over the disk, where the spinal open book on the boundary may have a more general spine rather than a multilink binding. The corresponding braided wiring diagrams then describe a more general surface arrangement, rather than DJVS immersed disks, see Remark~\ref{rem:disks}. We believe that the diagrammatic approach would 
be useful for constructions in this more general settings as well, as it provides tools to modify and control certain symplectic constructions via basic combinatorics.}    

\subsection{Nearly Lefschetz fibrations via monodromy factorization.}

{%A classical Lefschetz fibration is encoded by an ordered list of vanishing cycles (one for each Lefschetz critical point) up to Hurwitz moves, and the monodromy of the open book induced on the boundary is given as the product of positive Dehn twists about these vanishing cycles. 
	Similar to the classical case, a nearly Lefschetz fibration is encoded by an ordered list of vanishing cycles (corresponding to Lefschetz critical points) and vanishing arcs (corresponding to exotic fibers) up to \emph{generalized Hurwitz equivalence}~\cite{BaHa2}, and the monodromy of the spinal open book induced on the boundary is given as a product of positive Dehn twists about the vanishing cycles and \emph{boundary interchanges} about the vanishing arcs. The boundary interchange is a half-twist along an arc in the page $P$ connecting these two components of $\d P$, as in Figure~\ref{fig:squareroot}. (Recall this is the monodromy around the local model of an exotic fiber from Example~\ref{ex:squareroot}.) The group of diffeomorphisms of $P$ generated by Dehn twists and boundary interchanges is the spinal mapping class group $SMod(P)$.}

{
While a classical open book decomposition is abstractly encoded by its fiber and monodromy, determining a spinal open book requires further specifying the topology of the spine components as $S^1\times \Sigma_i$ for some surfaces with boundary $\Sigma_i$. It is important to note that for a given spinal open book with fiber $P$ and monodromy $\phi$, not every positive factorization of $\phi$ in $SMod(P)$ will correspond to a nearly Lefschetz fibration inducing the given spinal open book on its boundary. We must put further constraints on the exotic fibers to get the right topology for the spine components. For this reason, \cite[Definition 2.4]{HRW} introduces the notion of \emph{admissible} positive factorizations in $SMod(P)$, relative to a fixed choice of base surface $B$ and representation $\rho:\pi_1(B)\to Mod(P)$.
}

{
In our restricted setting of multilink binding, our spinal open books always have vertebrae $\Sigma_i$ which are disks, and thus the only possible base surface $B$ is a disk as well, and the representation mentioned above is trivial. Each spine component $\Sigma_i\times S^1$ intersects some collection of boundary components of $P$, which is denoted by $\partial^i_B(P)$. The number of these boundary components is the multiplicity of the corresponding binding component, denoted $d_i$. In our multilink binding case, the monodromy $\phi$ cyclically permutes the components of $\partial^i_B(P)$, so the collections $\partial^i_B(P)$ are actually determined by the monodromy $\phi$. The Riemann-Hurwitz formula implies that a degree $d_i$ branched covering from $\Sigma_i=D^2$ to $B=D^2$ has $d_i-1$ branch points. In this case, we obtain~\cite[Definition 2.4]{HRW} in our setting:}

{\begin{definition} \label{def:admissible}
	Given a spinal open book with multilink binding, its monodromy $\phi$ cyclically permutes subsets $\partial^i_B(P)$ consisting of $d_i$ boundary components of $P$. For this spinal open book, an \emph{admissible positive factorization} is a way of writing $\phi$ as a product of 
\begin{itemize}
	\item $d_i-1$ boundary interchanges between pairs of boundary components in $\partial^i_B(P)$
	\item positive Dehn twists about essential curves in $P$
\end{itemize}
\end{definition}
}

{Given a spinal open book with multilink binding, any admissible positive factorization of its monodromy $\phi$ corresponds to a nearly Lefschetz fibration such that the spinal open book induced on the boundary is the given one. Combining this with the results of~\cite{LVHMW2,HRW}, we obtain the following combinatorial characterization of minimal strong symplectic fillings of a contact manifold supported by a planar spinal open book with multilink binding.}

{
\begin{theorem}[\cite{LVHMW2,HRW}]\label{thm:disk-multisections}
		Suppose $(M,\xi)$ is a contact $3$-manifold supported by a planar spinal open book with multilink binding with fiber $P$ and monodromy $\phi$. Then the minimal strong symplectic filling of $(M,\xi)$ are in bijective correspondence with the admissible positive factorizations of $\phi$ up to generalized Hurwitz equivalence.
\end{theorem}
}

%Moreover, in the language of~\cite{HRW}, $D^2$ is the only possible base for this spinal open book, as all vertebra are disks.

\subsection{Braided wiring diagrams.} We introduce a diagrammatic way to encode a DJVS disk 
arrangement $\Gamma=\cup \Gamma_i$ in $D_x\times D_y$.  {(The same strategy works for somewhat more general immersed  surface arrangements with similar properties, without requiring that the components be disks, but we will work with DJVS arrangements to avoid lengthy definitions.)} The branch points of the projection arise when $\Gamma_i$ is tangent to a fiber of $\pi_x$, so we refer to these critical points of the projection as tangencies. The key information is the branch points, the intersections, and how these connect to each other. We will show that $\Gamma$ can be determined up to isotopy by a 1-dimensional singular braid with the tangencies and intersection points specified, called a \emph{braided wiring diagram}: 
our approach generalizes the notion of braided wiring diagram introduced in~\cite{Arvola} to allow for tangencies. Combinatorially, a braided wiring diagram can be encoded by a sequence $(b_0, S_1, b_1, S_2, \dots, b_{N-1},S_{N}, b_{N})$ where $b_i$ is a braid word on $n$ strands, and $S_i$ is a singularity--either a transverse intersection between a specified collection of adjacent strands or a tangency between two specified adjacent strands. (One could easily generalize this to allow more singularities, as long as a model is specified for each singularity type.) For brevity, we will also write a sequence of the form  $(b_0, S_1, b_1, S_2, \dots, b_{N-1},S_{N}, b_{N})$ as a word $b_0 S_1 b_1 S_2 \dots b_{N-1} S_{N} b_{N}$, omitting any identity elements $b_i$. 
	
We adopt the following notational conventions (conventions vary in the literature).  Strands are enumerated from  top to bottom  with local indices (i.e. distance from top at that position in the diagram). The diagram is read left to right. The singularities are labeled $T_i$ for a tangency between the $i^{th}$ and $i+1^{st}$ strands, and $I^i_{i+k}$ for a transverse intersection between the $i^{th}$ through $(i+k)^{th}$ strands.
	Braid words, read left to right, are in multiplicative notation, $\sigma_i$ indicating a crossing between the $i^{th}$ and $(i+1)^{st}$ strands where the $i^{th}$ strand crosses over the $(i+1)^{st}$ strand from left to right. See Figure~\ref{fig:brwiringdiagram} for an example.

\begin{figure}
	\centering
	\includegraphics[scale=.75]{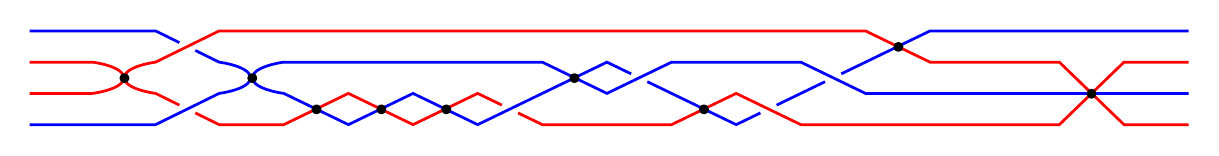}
	\caption{Braided wiring diagram $T_{2} \sigma_1^{-1}\sigma_3^{-1} T_{2} I^3_4 I^3_{4} I^3_{4} \sigma_3^{-1} I^2_{3}  \sigma_2^{-1} I^3_{4} \sigma_3\sigma_2 I^1_{2} I^2_{4}$,	
	or 
	$(1, T_{2}, \sigma_1^{-1}\sigma_3^{-1}, T_{2}, 1, I^3_4, 1, I^3_{4}, 1, I^3_{4}, \sigma_3^{-1}, I^2_{3}, \sigma_2^{-1}, I^3_{4}, \sigma_3\sigma_2, I^1_{2}, 1, I^2_{4})$ in expanded form.
	}
	\label{fig:brwiringdiagram}
\end{figure}

Given an arrangement $\Gamma$, we construct a braided wiring diagram as follows. Let $q_1,\dots, q_N\in D_x$ be the images of intersection points and tangencies (branch points) of $\Gamma$. Choose an embedding $\alpha:I\to D_x$ such that $\alpha(t_1)=q_1,\dots, \alpha(t_N)=q_N$ for $0<t_1<\cdots <t_N<1$ and $\alpha|_{[t_i-\delta,t_i+\delta]}(t) = (t-t_i)+q_i$ runs parallel to the real axis in $D_x$ for some sufficiently small $\delta>0$. Then $\pi_x^{-1}(\alpha(I))\cong [0,1]\times \C$, and $\Gamma\cap \pi_x^{-1}(\alpha(I))$ is a $1$-dimensional braid, except at singular points above $q_1,\dots, q_N$. This restriction of~$\Gamma$ over the preimage of the (piecewise linear) arc $\alpha$ is  the {\em braided wiring diagram} for $\Gamma$. This depends on the choice of $\alpha$, so the braided wiring diagram is not unique, but one can describe the moves induced on the braided wiring diagram by changes to $\alpha$ (singular Markov moves), 
cf~\cite{CS}. 

To go from the $1$-dimensional braided wiring diagram to an arrangement $\Gamma$, we will insert a local model in a neighborhood of each intersection and tangency and then connect the local models with a thickening of a braid. As braided wiring diagrams have not previously been used for arrangements with tangencies, we give a full description of the local model and interpretation of the braid monodromy near a tangency.
	
\subsection{Model for tangency.} \label{ss:models}  The standard model for a complex curve with a tangency to the projection $\pi_x$ is given by the curve $\{(x,y)\in \C^2\mid x=y^2\}$. The tangency occurs at the origin. In order to avoid a highly degenerate picture in our braided diagram, we choose a slightly perturbed model (which relates to the standard model by a simple complex coordinate change). Note that we are making a choice in our perturbation which fixes our conventions near the tangency.

Fix $0<\mu\ll \pi/2$. Our model for a tangency will be the complex curve $\{(x,y)\in \C^2\mid e^{i\mu}x = y^2\}$ with the tangency of $\pi_x$ at the origin. The braided wiring diagram and braid monodromy near this model are shown in Figure~\ref{fig:tangencymodel}. The preimage of the model over the interval in the $Re(x)$ axis includes the singularity, and gives the picture for the braided wiring diagram near a tangency. The braid monodromy about a small circle oriented counterclockwise around the singularity is a single positive half twist on two strands. We break up this circle into two half-circles showing the push-off of the singular braid in the positive and negative imaginary directions in $\C_x$.  {Pushing off in the positive imaginary direction gives the braid $\sigma_i^{-1}$ read left to right. Pushing off in the negative imaginary direction gives the trivial braid. The counterclockwise monodromy around the singularity is obtained by composing the inverse of the positive pushoff braid with the negative pushoff braid (since a counterclockwise circle travels right to left along the top and left to right along the bottom). Thus the total monodromy about the tangency is $\sigma_i$ as expected.} 

\begin{remark} We emphasize  that our choice of perturbed model with small $\mu>0$ sets a convention for the asymmetry between the positive and negative push-offs where the crossing (with respect to the projection to the $(Re(x),Re(y))$ plane) occurs in the positive push-off and not in the negative push-off. If we chose a perturbed model with $-\pi/2 \ll \mu <0$, the crossing would appear in the negative push-off instead (and the strand with negative slope would be the overstrand), though the total monodromy around the full circle would be the same. It is important to keep track of what happens on which side consistently with the model, so to avoid confusing conventions that might accidentally be mixed incorrectly, we assume all tangencies have the same model with $\mu>0$ as above so the crossing in the monodromy occurs in the positive (or ``back'') push-off of the wiring diagram.
\end{remark}

\begin{figure}
	\centering
	\includegraphics[scale=.3]{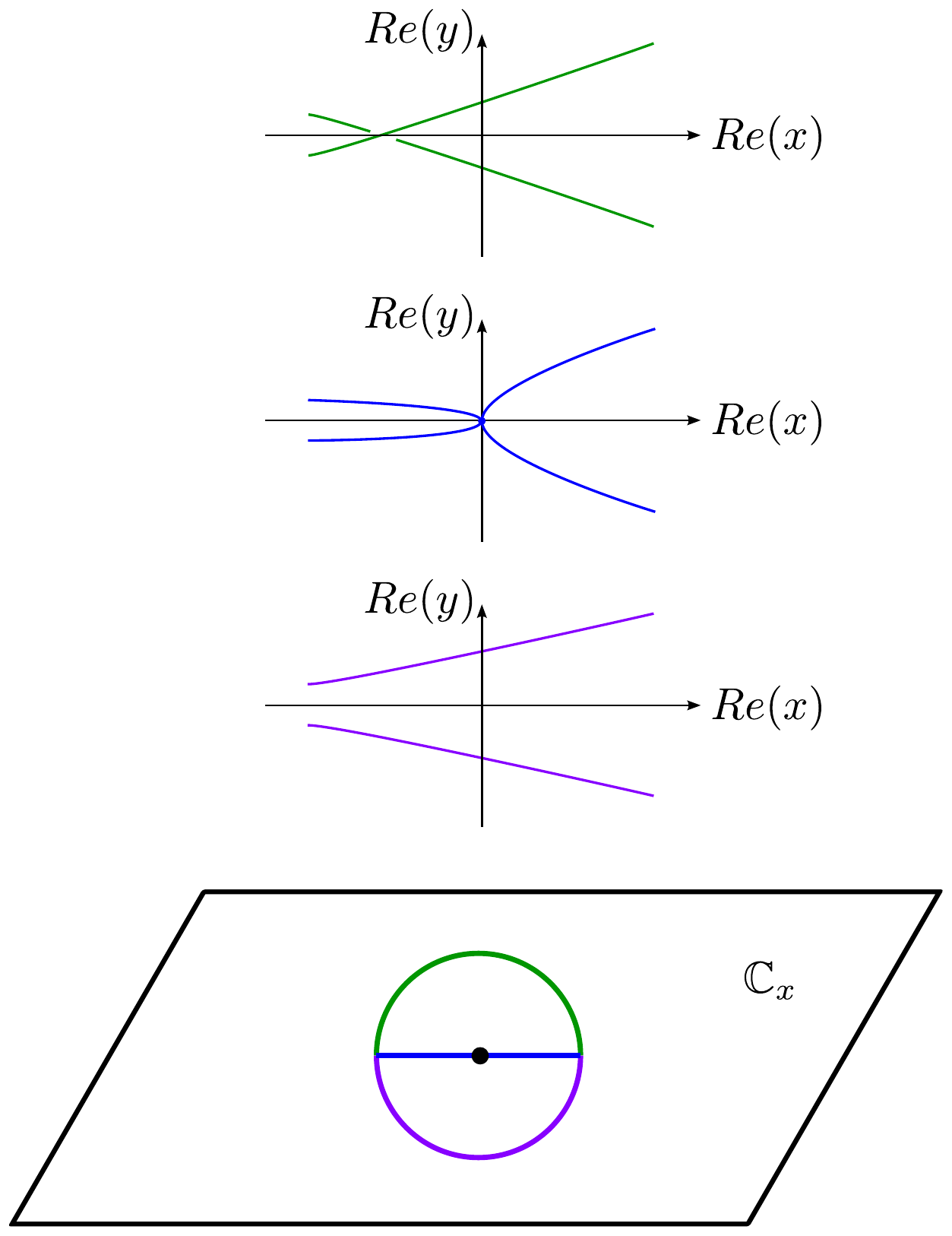}
	\caption{The model tangency $\{e^{i\mu}x = y^2\}$. The preimages of three curves in the base $\C_x$ are shown in their projections to the real parts of $x$ and $y$. We project out the imaginary $y$ axis, but record crossings as in the top curve where an overstrand indicates a larger value for the imaginary part of $y$. Note that in the singular diagram above $[-1,1]$ in the real axis of $\C_x$, the two strands on the right side of the singularity have imaginary part $0$, while on the left side, the upper strand has negative imaginary part and the lower strand has positive imaginary part.}
	\label{fig:tangencymodel}
\end{figure}

The model for a transverse intersection is given by a pencil of complex lines which are the complexification of real transverse lines. See~\cite{PS,CS} for more details on braided wiring diagrams with only transverse intersection singularities and models.  Figure~\ref{fig:brwiringdiagrampushoffs} shows a braided wiring diagram including both tangency and transverse intersection models, along with the pushoffs in the positive and negative imaginary directions. Let $\Delta_{j,j+k}$ denote the positive half twist on strands $j,j+1,\dots, j+k$ (e.g. $\Delta_{j,j+1}=\sigma_j$). Then the positive pushoff of $I^i_{i+k}$ is $\Delta_{i,i+k}^{-1}$ and the negative pushoff is $\Delta_{i,i+k}$.
Thus, for the surface bundle in the complement of the curve over a small circle around the image of the transverse intersection, the monodromy is $\Delta_{i,i+k}^2$, the positive Dehn twist about a curve convexly enclosing holes $i,i+1,\dots, i+k$.

\begin{figure}
	\centering
	\includegraphics[scale=.55]{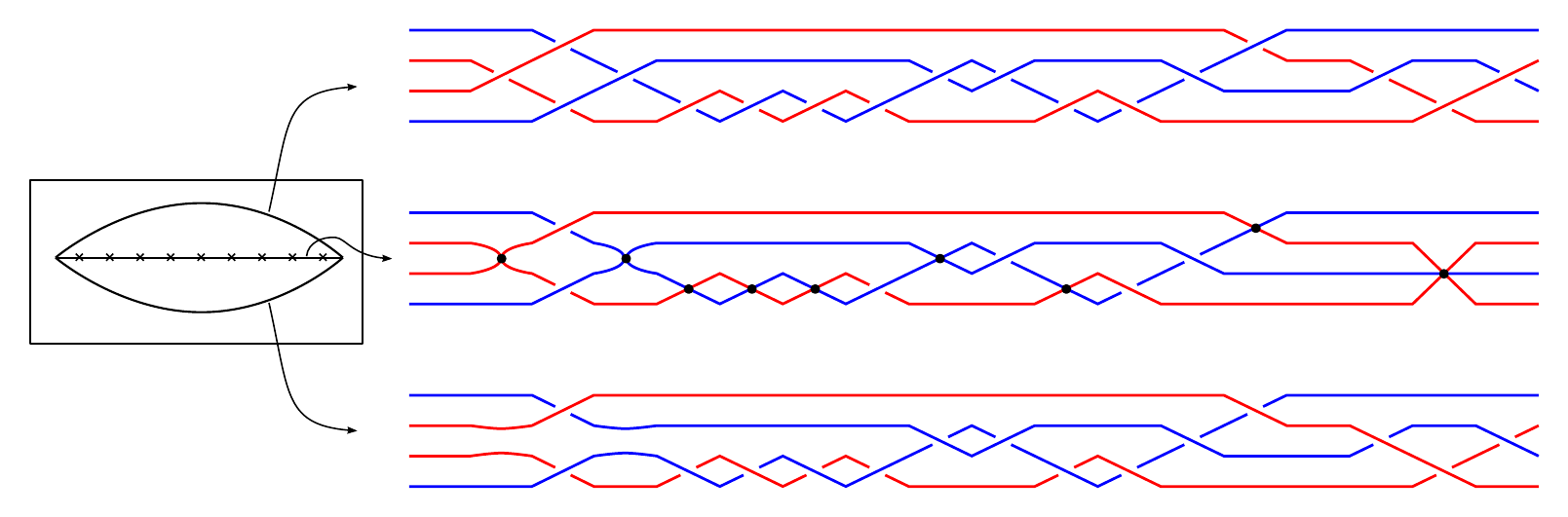}
	\caption{The braided wiring diagram from Figure~\ref{fig:brwiringdiagram}, along with its pushoffs in the positive (top) and negative (bottom) imaginary directions. To obtain the monodromy about the entire configuration, compose the inverse of the top braid with the bottom braid.}
	\label{fig:brwiringdiagrampushoffs}
\end{figure}

\subsection{Symplectic arrangement from braided wiring diagram.} Given a braided wiring diagram, we can construct an arrangement $\Gamma$ of symplectically immersed 
{surfaces}
in $D_x\times D_y$
whose restriction over the real axis in $D_x$ recovers the given diagram.  

\begin{lemma}\label{lem:wiringtocurves}
	Let $(b_0,S_1,b_1,\dots, S_{N},b_{N})$ be a braided wiring diagram. Then there exists an arrangement $\Gamma$ of symplectically immersed \op{surfaces} in $D_x\times D_y$ such that the braided wiring diagram of $\Gamma$ over the real axis in $D_x$ is $(b_0,S_1,b_1,\dots, S_{N},b_{N})$.
\end{lemma}

\begin{proof}
We extend the $1$-dimensional slice over the real axis by inserting the local models in a small neighborhood of each singularity. Then we extend the braid portions by a product with the imaginary $x$ direction, up to a minor isotopy near the endpoints of the braid to smoothly connect them to the local models. This fills in a surface over a neighborhood $\nu$ of the real axis in $D_x$, which we extend to the entirety of $D_x\times D_y$ by taking a product of each of the braids on the two components of $\partial \nu \setminus \partial D_x$ with the adjacent portion of the imaginary $x$ axis. If the resulting surface is not symplectic with respect to the product symplectic form $\omega_x\oplus \omega_y$, a smaller rescaling of it in the $D_y$ direction will be. This is because rescaling in the $D_y$ direction causes the $\omega_x$ term to dominate, and the tangent space of $\Gamma$ projects surjectively to the tangent space of $D_x$ with positive orientation away from the tangency singularities. Near the tangency singularities, the model is symplectic since the product symplectic form is compatible with the standard complex structure. Rescaling the $y$ direction preserves the symplectic property in the model so the surface remains symplectic here under this rescaling.
\end{proof}

{
\begin{remark}\label{rem:surfaces}
The surfaces constructed in the previous lemma may a priori have higher genus, so we might get a more general ``immersed multisection'', not necessarily  an immersed disk arrangement. The Euler characteristic and the number of boundary components of each surface component
can be easily computed from the diagram, in particular, we can check whether every component is a disk if we are interested in DJVS arrangements.
\end{remark}
}

\subsection{Scott deformation.} \label{ss:Scott} Every sandwiched singularity has a distinguished \emph{Artin smoothing} (or Artin component of the deformation space). The corresponding deformation of the weighted curve germ is the \emph{Scott deformation} which is obtained as follows. Given the curve germ $C$ with isolated singularity, blow up at the singular point. Next, deform the exceptional component of the total transform (via an arbitrarily small deformation) so that the exceptional curve intersects the proper transform of $C$ generically transversally at smooth points of the proper transform. The proper transform will generally have singular points remaining. At each such singular point, we perform the same process of blowing up and then taking an arbitrarily small deformation of the exceptional component so it intersects generically. Iterate this process until the proper transform has no further singularities, and then blow down all of the exceptional spheres. (Note that all exceptional spheres are disjoint from each other since we assumed that each exceptional sphere intersects the proper transform only at smooth points, and further blow-ups only occur at remaining singular points.) Thus the resulting curve in $\C^2$ only has transverse multipoint singularities. This result is by definition the Scott deformation of $C$. Note that the Scott deformation for a general $(C,w)$ may have branch points with respect to the projection $\pi_x$ in addition to the multipoint singularities. However, because the deformations of the exceptional curves can be made arbitrarily small, there is always a wiring diagram for the Scott deformation with \emph{no braiding} between the singular points. For example, the curve germ which had a deformation with wiring diagram as in Figure~\ref{fig:brwiringdiagram} is two components each with simple cusp ($x^2=y^3$) singularities such that the total multiplicity of intersection between the two components is $7$. The (unbraided) wiring diagram for the Scott deformation of this germ is shown in Figure~\ref{fig:Scott}.

\begin{figure}
	\centering
	\includegraphics[scale=.75]{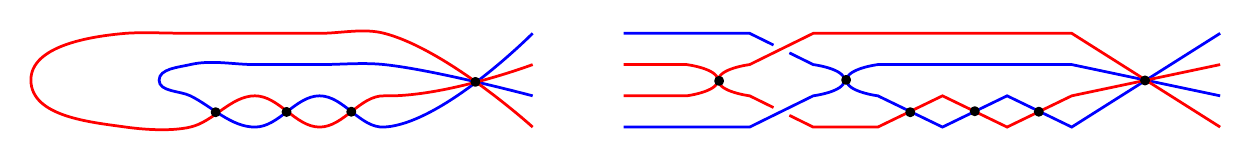}
	\caption{Scott deformation of a germ with two irreducible components, each with a simple cusp singularity, and intersecting each other with multiplicity $7$. On the left is a picture of the real part of the Scott deformed curve. On the right is a braided wiring diagram for the Scott deformed curve. Another deformation of this germ is shown via a braided wiring diagram in Figure~\ref{fig:brwiringdiagram}.}
	\label{fig:Scott}
\end{figure}

\subsection{From wiring diagram data to vanishing cycles/arcs.}
Given a DJVS arrangement $(\Gamma, \{p_j\})$, Theorem~\ref{thm:smooth-nearlyLefschetz} gives a construction of an associated manifold $W$ with a nearly Lefschetz fibration. We have seen that the arrangement can be encoded in a braided wiring diagram (with marked points). The next lemma explains how to 
read off the vanishing cycles and vanishing arcs of the nearly Lefschetz fibration directly from the braided wiring diagram. 
 The vanishing cycles and vanishing arcs will be determined up to {\em generalized} Hurwitz equivalence, \cite{BaHa2}, 
\cite[Definition 2.14]{HRW}: in addition to the classical Hurwitz moves, there is a {\em framing conjugation} move that consists of conjugating a boundary interchange along an arc by a Dehn twist around a boundary parallel curve.
It is important to note  that the braid monodromy is an element of the mapping class group 
	$Mod(\text{disk with punctures})$ while the monodromy of the nearly Lefschetz fibration is an element of $Mod(\text{disk with holes})$. The inclusion induces the map  $Mod(\text{disk with holes}) \to Mod(\text{disk with punctures})$, which sends the monodromy of the Lefschetz fibration to the braid monodromy of the corresponding arrangement. We will use the braids to get partial information on the monodromy of the fibration, together with local models that pin down the boundary parallel Dehn twist that are in the kernel of the above map. Despite this subtlety, we will use similar notation for braids (considered as elements of the mapping class group of the punctured disk) and their lifts the mapping class group of the compact fiber $P$, a disk with holes.

	{The following lemma works for an arbitrary braided wiring diagram with marked points, where all the intersections are marked. We state it for DJVS arrangements because we are mainly interested in those.} 
\begin{lemma} \label{lem:vanishingfromwiring}
	Let $(\Gamma,\{p_j\})$ be a DJVS arrangement with marked points and the corresponding  braided wiring diagram $(b_0,S_1,b_1,S_2,\dots,b_{N-1},S_{N},b_{N})$, {where a free marked count will also be considered a singular point $S_j$, treated as an intersection point with only one strand.}
	Fix an identification of the planar fiber $P$, a disk with 
	disjoint holes of equal radius centered along a straight line. If the $j^{th}$ singularity is an intersection $I^i_{i+k}$, let $A_j$ denote a convex curve containing holes $i,i+1,\dots, i+k$, and let $\phi_{j} = \Delta_{i,i+k}^{-1}$ denote the {negative half twist}. (In the case of a free marked point where $k=0$, $A_j$ is the boundary parallel curve around the $i^{th}$ hole, {$\phi_j:=\Delta_{i,i}^{-1}$ is the negative half twist on $A_j$}). If the singularity $S_j$ is a tangency between strands $i,i+1$, let $A_j$ denote a straight line arc connecting hole $i$ and $i+1$, and 
	let $\phi_{j}:=\Delta_{i,i+1}^{-1}=\sigma_i^{-1}$ be the negative half twist along $A_j$.
	 
	Let $(V_1,V_2,\dots, V_{N})$ be an ordered collection of vanishing cycles and arcs corresponding to the nearly Lefschetz fibration on the manifold $W_{(\Gamma, p)}$ produced by Theorem~\ref{thm:smooth-nearlyLefschetz}. Then up to {generalized} Hurwitz equivalence, 
	\begin{eqnarray*}
		V_1 &=& b_0^{-1}(A_1)\\
		V_2 &=& b_0^{-1}\circ \phi_1^{-1}\circ b_1^{-1}(A_2)\\
		 & \vdots & \\
		V_{N}&=& b_0^{-1}\circ \phi_1^{-1}\circ b_1^{-1}\circ \cdots \circ \phi_{N-1}^{-1}\circ b_{N-1}^{-1}(A_{N}).
		\end{eqnarray*}
		{Above, $b_i \in Mod(P)$ stands for an arbitrary lift of the braid $b_i$ to the mapping class group of the disk with holes. Different choices yield the same vanishing cycles and may change vanishing arcs by framing conjugation only.}
\end{lemma}

\begin{proof}
	Recall that $W_{(\Gamma, p)}=[(D_x\times D_y)\#_n \cptwobar]\setminus \cup \nu(\widetilde{\Gamma}_i)$, and the Lefschetz fibration is the composition of the blow-down with the projection to $D_x$. Let $\alpha:I\to D_x$ be the arc over which lies the braided wiring diagram with data $(b_0,S_1,b_1,S_2,\dots,b_{N-1},S_{N},b_{N})$. We will define vanishing cycles and arcs in the fiber over $\alpha(0)$. We can identify the vanishing cycles by choosing an ordered collection of loops, one about each singular value, such that the composition of the loops is homotopic to an embedded loop convexly surrounding all the singular values. The monodromy about each loop will be a positive (right handed) Dehn twist about the vanishing cycle or a positive boundary interchange about the vanishing arc, depending on whether the singularity is Lefschetz or exotic respectively. Thus if we can choose such a collection of loops in $D_x$, determine the monodromy above each loop, and realize this monodromy as a positive Dehn twist about a curve, then we will have identified a set of vanishing cycles for the Lefschetz fibration. {The choices made in this process may lead to different collections of vanishing cycles/arcs but any two such sets of data are related by generalized Hurwitz equivalence.}

	Let $t_1,\dots, t_{N}\in (0,1)$ be such that $q_i:=\alpha(t_i)$ is the image of the singularity $S_i$. Let $t_0=0$, $t_{N+1}=1$. Choose a sufficiently small $\delta>0$ such that
	$\alpha|{[t_i-\delta,t_i+\delta]}(t)={(t-t_i)}+q_i$ is parallel to the real axis in $D_x$ and the braids $b_i$ lie over the intervals $\alpha([t_i+\delta, t_{i+1}-\delta])$ for $i=0,\dots, N$. Let {$h_i(t)=\delta e^{\pi i(1-t)}+q_i$} be the {clockwise} path along the {upper} semi-circle of radius $\delta$ around $q_i$. Let $g_i(t) = \delta e^{{2\pi}it} +q_i$ be the counterclockwise circle of radius $\delta$ around $q_i$. Then the loop $l_i$ is the concatenation of the following paths:
	$$l_1 = \alpha|_{[0,t_1-\delta]} \star g_1 \star \alpha|_{[0,t_1-\delta]}^{-1}$$
	and 
	$$l_i = \alpha|_{[0,t_1-\delta]}\star h_1 \star \alpha|_{[t_1+\delta,t_2-\delta]} \star \cdots \star h_{i-1} \star \alpha|_{[t_{i-1}+\delta,t_i-\delta]}\star g_i \star \alpha|_{[t_{i-1}+\delta,t_i-\delta]}^{-1} \star h_{i-1}^{-1}\star \cdots \star \alpha|_{[t_1+\delta,t_2-\delta]}^{-1} \star h_1^{-1} \star \alpha|_{[0,t_1-\delta]}^{-1}$$
	for $i=2,\dots, N$. See Figure~\ref{fig:loops}.
	
	\begin{figure}
		\centering
		\includegraphics[scale=.5]{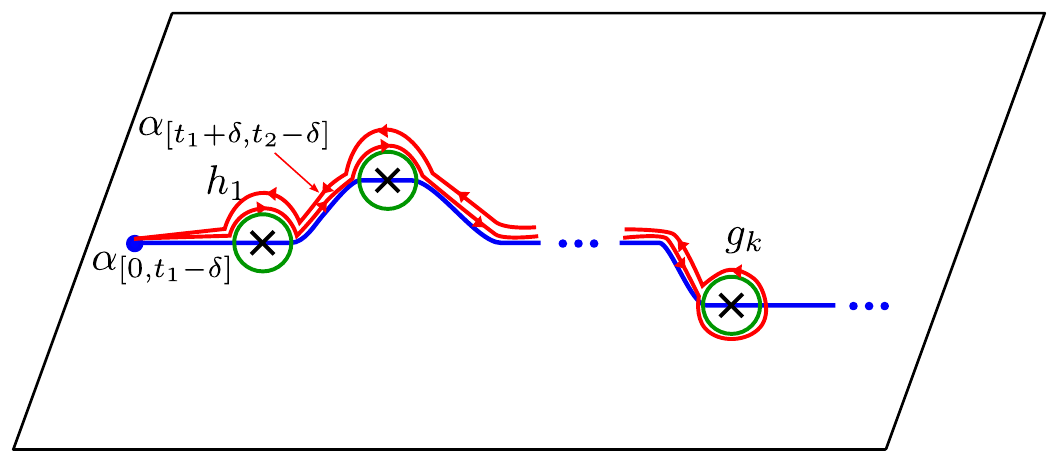}
		\caption{Loops to determine monodromy about vanishing cycles/arcs in $\C$, broken down into arcs with the corresponding monodromy for each arc specified.}
		\label{fig:loops}
	\end{figure}
	
	 We will compute the monodromy over the loops $l_i$ as the composition of the monodromy over the elementary open paths that we decomposed $l_i$ into. The monodromy over each of these open paths is well-defined because we can identify different fibers using the projection from $\C^2$ to $\C$ onto the second coordinate, and we will assume (by an isotopy if needed) that the surface intersects the fibers at the start and end of these paths in the same set of points under this identification. The corresponding elementary monodromy is determined via local models in each case. By definition, the monodromy over $\alpha|_{[0,t_1-\delta]}$ is $b_0$ and the monodromy over $\alpha|_{[t_i+\delta,t_{i+1}-\delta]}$ is $b_i$ for $i=1,\dots, N-1$. As in the statement of the lemma, we will fix arbitrary lifts of these braids to the mapping class group of the fiber $P$, keeping the same notation.  From a local model for the intersection of a pencil of transverse complex lines, we see that the braid monodromy over $g_i$ is the full twist $\Delta_{j,j+k}^2$ on the corresponding strands. For the Lefschetz fibration on the complement of the tubular neighborhoods of the strict transforms of the curves in the blowup, the local model as in \cite[Lemma 3.2]{PS} gives the the monodromy over $g_i$ as the positive Dehn twist $\tau_{A_i}$ along the convex simple closed curve $A_i$ enclosing the corresponding holes.
	 For the free marked point, we get a positive Dehn twist around a single hole. In the case of a tangency,  the local model gives the braid monodromy over $g_i$ as the half twist $\Delta_{j,j+1}$. The monodromy of the nearly Lefschetz fibration is the positive boundary interchange $t_{A_j}$ on the arc $A_j$, by  the local model as in Figure~\ref{fig:squareroot}. As explained before the statement of the lemma, the local models for the fibrations determine the boundary parallel Dehn twists in the monodromy, in addition to the information read off from the braid monodromy of the arrangement.
	
	{Similarly, the local model for the intersection of complex lines tells us that the braid monodromy over the path $h_i$ is the negative half twist $\phi_j =  \Delta_{i,i+k}^{-1}$ on the corresponding strands, as defined in the statement of the lemma. (The negative sign is due to the clockwise orientation of $h_i$.) In the case of tangency between strands $i$ and $i+1$, the braid monodromy over the path $h_i$ is the negative half twist $\phi_j =  \sigma_i^{-1}$: we use the local model and pay attention to the path orientation and the fact that it goes ``in the back'' of the diagram.  For the nearly Lefschetz fibration on the complement of the tubular neighborhoods of the strict transforms of the curves in the blowup, we can determine the boundary Dehn twists from the local models. However, we will see that the maps $\phi_j$ only enter into the total monodromy expression as conjugations. The boundary Dehn twists will not affect the generalized Hurwitz equivalence class of the fibration, therefore, we can choose to ignore them at this step.}

	As above, let $\tau_C$ denote a positive Dehn twist about a curve $C$ and let $t_\gamma$ denote a boundary interchange about an arc $\gamma$. Then the monodromy $m_{l_i}$  of the nearly Lefschetz fibration $W_{(\Gamma, p)}$
	about $l_i$ is given by
	$$m_{l_1} = b_0^{-1}\circ \tau_{A_1} \circ b_0$$
	when $S_1$ is an intersection, and by
	$$m_{l_1} = b_0^{-1} \circ t_{A_i} \circ b_0$$
	when $S_1$ is a tangency.
	
	Similarly 
	$$m_{l_i} = b_0^{-1}\circ \phi_{1}^{-1} \circ b_1^{-1} \circ \cdots \circ \phi_{{i-1}}^{-1}\circ b_{i-1}^{-1}\circ \tau_{A_i} \circ b_{i-1}\circ \phi_{{i-1}} \circ \cdots \circ b_1 \circ \phi_{1} \circ b_0$$
	when $S_i$ is an intersection, or
	$$m_{l_i} = b_0^{-1}\circ \phi_{1}^{-1} \circ b_1^{-1} \circ \cdots \circ \phi_{{i-1}}^{-1}\circ b_{i-1}^{-1}\circ t_{A_i} \circ b_{i-1}\circ \phi_{{i-1}} \circ \cdots \circ b_1 \circ \phi_{1} \circ b_0$$
	when $S_i$ is a tangency.
	
	Finally, we note that for any diffeomorphism $\phi$ of the fiber surface, $\phi\circ \tau_C \circ \phi^{-1} = \tau_{\phi(C)}$ and $\phi\circ t_{\gamma} \circ \phi^{-1} = t_{\phi(\gamma)}$. This is well known and commonly used in the literature in the case of the Dehn twist about a closed curve $C$. It can be verified for the boundary interchange by observing that $\phi\circ t_{\gamma} = t_{\phi(\gamma)}\circ \phi$ as in Figure~\ref{fig:conjugationtwist}.
	
	\begin{figure}
		\centering
		\includegraphics[scale=.5]{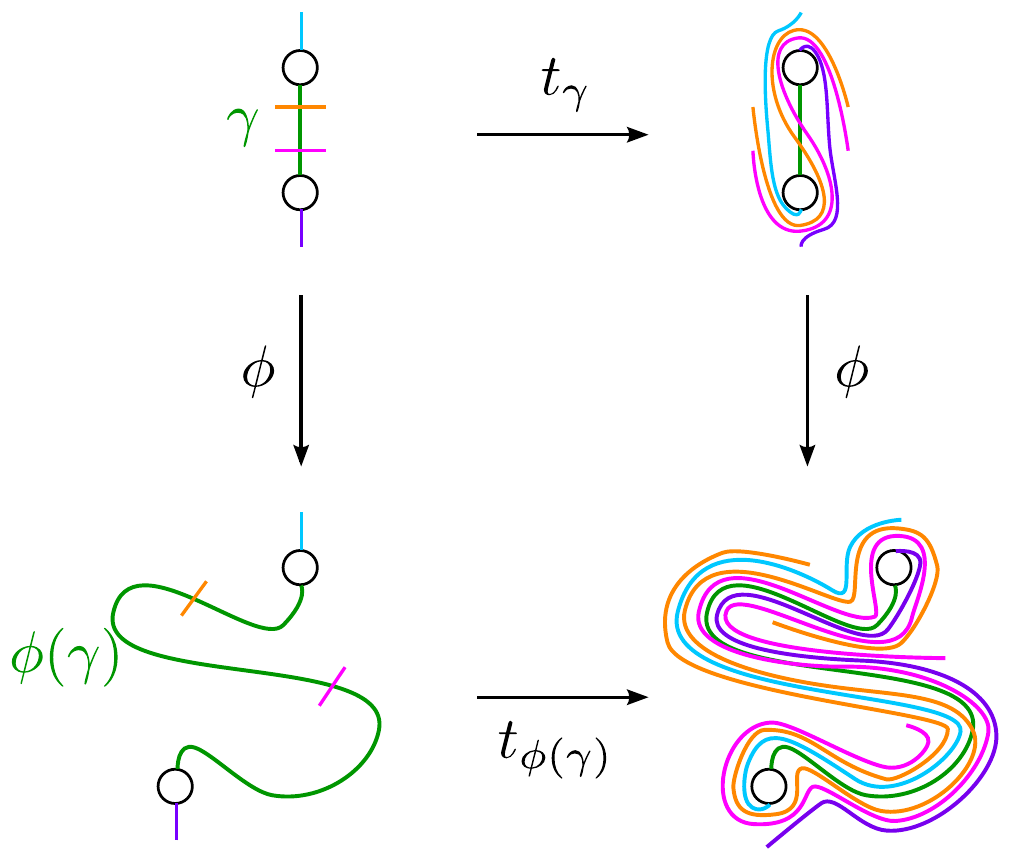}
		\caption{Illustration that for any diffeomorphism $\phi$ of a surface and arc $\gamma$ connecting two boundary components, $\phi\circ t_{\gamma} = t_{\phi(\gamma)}\circ \phi$.}
		\label{fig:conjugationtwist}
	\end{figure}
	
	Therefore $m_{l_i} = \tau_{V_i}$ or $t_{V_i}$ where $V_i$ is the curve or arc given in the statement of the lemma. Our constructions had an arbitrary choice of lifts of braid elements $b_i$ and $\phi_i$ to the mapping class group of the disk with holes, with an undeterminacy coming from boundary parallel Dehn twists. In the formulas we obtained, the boundary parallel Dehn twists commute with the Dehn twists along the vanishing cycles, so the vanishing cycles are not affected by this undeterminacy. Vanishing arcs are changed by  the framing conjugation move: a boundary interchange $t_\gamma$ is replaced by
	$\tau^{-1}_\delta t_\gamma \tau_\delta$, conjugating by the Dehn twist around a boundary parallel curve. This gives the same generalized Hurwitz equivalence class, \cite{BaHa2}, \cite[Definition 2.14]{HRW}.
	\end{proof}

\subsection{From vanishing cycles/arcs to wiring diagrams} 

{A nearly Lefschetz fibration is encoded by a monodromy factorization of the spinal open book, expressed as an ordered product of positive Dehn twists $\tau_{V_i}$ and boundary interchanges $t_{V_i}$.  The next lemma shows how to translate the factorization data into a braided wiring diagram. Note that at this point, we do not keep track of the topology of the corresponding multisection in $D^2 \times D^2$. A priori, the diagram might not produce a configuration of immersed disks. See Remarks~\ref{rem:surfaces} and~\ref{rem:disks}.} 
		{\begin{lemma} Given a minimal nearly Lefschetz fibration on the complement of a multisection in a blowup of $D^2 \times D^2$, encoded by its vanishing cycles and vanishing arcs $V_i$ as above, one can produce an explicit braided wiring diagram 
		$(b_0,S_1,b_1,S_2,\dots, b_{N-1},S_N,b_N)$, possibly with free marked points, 
		such that the Lefschetz fibration data can be recovered from the wiring diagram as in Lemma~\ref{lem:vanishingfromwiring}. 
		\end{lemma}
The construction will be given in the proof.}
		\begin{proof}
		We will inductively build a braided wiring diagram by determining braids $b_i$ and convex curves surrounding a collection of adjacent holes or straight line arcs connecting a pair of adjacent holes $A_i$ such that 
		\begin{eqnarray*}
			V_1 &=& b_0^{-1}(A_1)\\
			V_2 &=& b_0^{-1}\circ \phi_1^{-1}\circ b_1^{-1}(A_2)\\
			& \vdots & \\
			V_{N}&=& b_0^{-1}\circ \phi_1^{-1}\circ b_1^{-1}\circ \cdots \circ \phi_{N-1}^{-1}\circ b_{N-1}^{-1}(A_{N}).
		\end{eqnarray*}
		
		Let $P$ be the fiber of the planar spinal open book, identified with a disk with $r$ holes of equal radius centered on the real axis and labeled $1,\dots, r$ in order of increasing real part. 
		As above, there is a surjective map $\iota: Mod(P) \to B_r$ from the mapping class group of $P$ to the braid group, with kernel generated by the boundary parallel Dehn twists.

		If $V_1$ is a vanishing cycle (closed curve), choose a diffeomorphism $\psi_0:P\to P$ such that $\psi_0(V_1)$ is convex and encloses holes $i,\dots, i+k$. Let $b_0:=\iota(\psi_0)$. Let $A_1 := \psi_0(V_1)$, $S_1:=I^i_{i+k}$ marked with a point $p_1$, and  $\phi_1:=\Delta_{i,i+k}^{-1}$. Note that if $A_1$ only encloses a single hole, $S_1$ is not a singularity, but we will still mark it with a point $p_1$. If $V_1$ is a vanishing arc, choose a diffeomorphism $\psi_0:P\to P$ such that $\psi_0(V_1)$ lies on the real axis of the disk and connects two adjacent boundary components, holes $i$ and $i+1$. Then let $b_0 := \iota(\psi_0)$, $A_1 := \psi_0(V_1)$, $S_1:=T_i$, and $\phi_1 := \sigma_i^{-1}$.
		
		Inductively, we will define $b_{l-1}$, $A_l$ and $\phi_l$ from previously defined values of $b_{j-1}$, $A_j$, $\phi_j$ for $j<l$. Choose a diffeomorphism $\psi_{l-1}:P\to P$ such that $A_l:=\phi_{l-1}\circ b_{l-2} \circ \cdots \circ b_1 \circ \phi_1 \circ b_0(V_l)$ is a convex curve or straight line arc on the real axis. Let $b_{l-1}=\iota(\psi_{l-1})$. If $A_l$ is a closed curve containing holes $i,\dots, i+k$, let $S_l:=I^i_{i+k}$, marked with a point $p_l$, and $\phi_l:=\Delta_{i,i+k}^{-1}$. Again, if $A_l$ happens to only enclose a single hole, $S_l$ will not be a singularity, but we will still mark the point $p_l$. If $A_l$ is an arc connecting the $i$ and $i+1$ strands, let $S_l:=T_i$, and let $\phi_k:=\sigma_i^{-1}$. We let $b_N=id$.
		
		This produces a braided wiring diagram $(b_0,S_1,b_1,S_2,\dots, b_{N-1},S_N,b_N)$. By Lemma~\ref{lem:wiringtocurves}, there is an arrangement $\Gamma$ of symplectically immersed disks in $D_x\times D_y$ realizing this braided wiring diagram. The points $p_1,\dots, p_N$ on the braided wiring diagram extend to decorations on $\Gamma$. Thus Theorem~\ref{thm:smooth-nearlyLefschetz} produces a manifold $W_{(\Gamma, p)}$ supporting a nearly Lefschetz fibration. By Lemma~\ref{lem:vanishingfromwiring}, by construction, the vanishing cycles and arcs of this nearly Lefschetz fibration are exactly $V_1,\dots, V_N$. 		
		%Since the given symplectic filling is supported by a nearly Lefschetz fibration with the same ordered list of vanishing cycles and arcs, we can identify $W_{(\Gamma, p)}$ with the given filling, up to symplectic deformation equivalence.
	\end{proof}
	{\begin{remark}\label{rem:disks} As in Remark~\ref{rem:surfaces}, the braided wiring diagram we construct might not necessarily encode a DJVS immersed disk arrangement, as the corresponding multisection components may have more complicated topology. However, if the nearly Lefschetz fibration is a filling of a spinal open book with multilink binding, then the corresponding
	monodromy factorization must be admissible as in Defintion~\ref{def:admissible}. The combinatorics of the boundary interchanges in this case implies that braided wiring diagram 
	yields a DJVS disk arrangement. 
	\end{remark}
}
	
			\section{Distinguishing the fillings} \label{sec:distinguish}
	
	Constructions of the previous sections encode  symplectic fillings of sandwiched links in terms of immersed disks arrangements. We now explain how to distinguish some of these fillings from one another, and how to show, in certain cases, that a given filling is unexpected (that is, not {homeomorphic} to any Milnor fiber). 
	
	\subsection{Detecting the incidence matrix}
	
	To distinguish fillings, we use the approach of N\'emethi--Popescu-Pampu \cite{NPP}, who showed that certain homological data of a filling (completed to a closed 4-manifold with a cap) is detected by the incidence matrix of the curvettas of the corresponding picture deformation. Although they work in the complex algebraic setting, their proof is topological and thus adapts easily for the DJVS immersed disk arrangements, c.f.~\cite{PS}. In~\cite{NPP} and~\cite{PS}, the statements are given in the smooth category; however, since the corresponding smooth manifolds are distinguished by homological tools, the arguments (if stated carefully) imply that the fillings are not homeomorphic. 
	
	Following~\cite{NPP}, we first distinguish the fillings relative to a fixed marking of the plumbing structure on the boundary. Note that since every homeomorphism of a $3$-manifold is isotopic to a diffeomorphism, orientation-preserving homeomorphisms preserve the plumbing structure, up to topological isotopy. 
	If $W$ and $W'$ are two Stein fillings of $(Y,\xi)$, we say they are \emph{strongly homeomorphic} if there is a homeomorphism $\phi:W \to W'$ such that $\phi$ takes the marking of  $\partial W\cong Y$ to the marking of $\partial W'\cong Y$.  After an isotopy, we can assume that $\phi$ is smooth in a collar neighborhood of $\d W$.
	The hypothesis of the fixed  boundary marking can often be removed, depending on the mapping class group of link (the latter depends on the symmetries of the plumbing graph, see Section~\ref{sec:top-plumbing}).  After addressing the relative case, we will see how to distinguish fillings in the absolute sense, at least under favorable conditions.

	\begin{remark} \label{rem:markingOBD}
	In~\cite{PS}, we described the boundary data in terms of a fixed open book rather than a marking. Because of Lemmas~\ref{OBisotopy} and \ref{binding-from-fibration}, fixing the marking is equivalent to fixing the isotopy class of 
	the spinal open book, if one also knows the multiplicities of the curvettas $C_i$ at $0$. Indeed, by Lemma~\ref{binding-from-fibration}, the marking gives binding components $\beta_i$ for $i=1, \dots, m$ that correspond to the curvettas $C_i$. The remaining  outer binding component $\beta_0$ is then uniquely determined as well:  the auxiliary $(-1)$ vertices adjacent to $E(C_i)$ determine the blow downs of the augmented resolution graph, and therefore the vertex $E_0$.     
	\end{remark}

	For a DJVS arrangement $(\Gamma, p)$ with $\Gamma  =\cup_{i=1}^m
	\Gamma_i$, the set of marked points $p= \{p_j\}_{j=1}^n$ contains all intersection points between the $\Gamma_i$'s and possibly a number of free points. 
	The {\em incidence matrix}
$\mathcal{I}(\Gamma, \{ p_j\})$ has $m$ rows and $n$ columns, 
defined so that its entry $a_{ij}$ at the intersection of $i$-th row and $j$-th column equals the multiplicity of the marked point $p_j$ on  $\Gamma_i$. There is no 
canonical labeling of the points $p_j$, so the incidence matrix is defined only up to permutation of columns. We say that two arrangements 
$(\Gamma, \{ p_j\})$ and $(\Gamma', \{ p'_j\})$ are {\em combinatorially equivalent} if their incidence matrices coincide (up to permutation of columns, i.e. up to relabeling of 
the marked points).

\begin{prop}(\cite[c.f. Theorem 4.3.3]{NPP})\label{incidence-matrix} Let $(Y, \xi)$ be the contact link of a sandwiched singularity, and fix the marking on $Y$ induced by $(C, w)$ as above.  Let   $W$ and $W'$ be two strongly homeomorphic Stein fillings that arise from two DJVS  immersed disk 
arrangements $(\Gamma, \{ p_j\})$ resp. $(\Gamma',\{p_j'\})$ compatible with the same decorated germ $(C, w)$.  
Then the incidence matrices $\mathcal{I}(\Gamma, \{ p_j\})$ and $\mathcal{I}(\Gamma', \{ p_j'\})$ are equal, up to permutation of columns.
\end{prop}
	
\begin{proof} We only briefly sketch the proof since it is very similar to that in~\cite{NPP} which addresses the case of Milnor fibers. We restate the argument slightly to emphasize that it works up to a homeomorphism despite being stated for diffeomorphisms
in~\cite{NPP, PS}. 
	We find it convenient to invoke open books in our setting, cf~\cite[Section 6.3]{PS}.
If there is a strong homeomorphism $\phi$ that gives an identification $\partial W \cong \partial W' \cong Y$ as above, by Remark~\ref{rem:markingOBD} we can assume the open books on $\partial W$ and $\partial W'$ induced by the Lefschetz 
fibration structures on  $W$ and $W'$ are mapped to each other by $\phi$. 

The incidence data of the arrangements can be recovered from $W$ and $W'$ by closing the fillings with the cap described at the end of Section~\ref{sandwich-setup}: the cap $U$ will be same for both $W$ and $W'$, and $W \cup U$ and $W' \cup U$ are both diffeomorphic to  $\#_{i=1}^n \cptwobar$. We analyze how the cap is attached to the boundary of the filling.    The cap $U$ can be thought of as a $2$-handlebody, with a single $0$-handle and a $2$-handle attached for each component $C_1,\dots, C_m$. Let $K=K_1,\dots, K_m$ denote the belt spheres of these $2$-handles, framed by the handle structure (where the framing is represented by a parallel belt sphere, meaning the boundary of a pushoff of the co-core which is disjoint from the co-core). Turning this handle decomposition upside down, $U$ has $m$ $2$-handles, attached to $\d U$ along the framed link $K \subset \d U$, and a single $4$-handle. Via the gluing map, $K$ is identified with a link in $\partial W$. In the handle decomposition of $U$, the $2$-handles are the tubular neighborhoods of the disks
 $\tilde{\Gamma}_i$, so that the link component $K_i$ is isotopic in $\partial W$ to a boundary component of the page $P$ that corresponds to $C_i$. To capture the framing, consider the nearly Lefschetz fibration 
 on $W$ from Theorem~\ref{thm:smooth-nearlyLefschetz}.
 Fix a {generic} choice of $x_0 \in D_x$, and assume that $\pi_x(K_i)=x_0$ for the chosen copy of $K_i$. By construction of 
 $W\cup U$, the co-core of the corresponding 2-handle $T_i$ of $U$ is the disk bounded by $K_i$ that lies 
 in the complement of {$W$}  in $\{x_0\} \times B_y$ and intersects $C_i$ at one point. 
 It follows that the framing of $K_i$ in $\d W$ is given by the parallel copy of that boundary component in a nearby regular fiber of the nearly Lefschetz fibration. We can isotope this curve to lie in a fiber over $\partial B_x$, showing that the framing is the open book page framing on $K_i$. The upshot is that for the fixed spinal open book on
 $\d W$, the capped-off $4$-manifold $W \cup U \cong \#_{i=1}^n \cptwobar$ is formed by attaching $2$-handles to the framed link $K$ in $\d W$  isotopic to the binding of the open book, where the framings on the components $K_i$ are determined by the open book, and then adding a $4$-handle. Since the strong homeomorphism $\phi$ is smooth in a collar neighborhood of $\d W$ and carries the fixed spinal open book on 
 $\d W$ to the fixed spinal open book on $\d W'$, we see that $\phi$ extends {by the identity on $U$} to a homeomorphism $W \cup U$ to $W' \cup U$ carrying handles of $U$ in $W \cup U$ to the handles of $U$ in $W' \cup U$ (in fact, $\phi$ restricts to $U$ as a diffeomorphism).  
 In the manifold  $W \cup U = \#_{i=1}^n \cptwobar$, 
 we consider the spheres $\tilde{\Sigma}_i \subset U$  given by the union of the germ components $C_i$ capped-off by the cores of the handles of the cap $U$. {Since the extension of $\phi$ is the identity on $U$,} %It follows that $\phi$  
 it carries these spheres to the corresponding spheres  ${\tilde{\Sigma}}'_i$ in $W' \cup U$.

As in \cite{NPP}, 
there is a unique basis $\{e_j\}_{j=1}^n$ for $H_2(\#_{i=1}^n \cptwobar)$ {of classes of square $-1$} such that the intersection numbers $[\tilde{\Sigma}_i] \cdot e_j$ are all positive; these classes are given by the exceptional spheres in $\#_{i=1}^n \cptwobar$, which in turn correspond to the marked points $p_j$.  On the other hand, these intersection numbers encode the incidence matrix of $(\Gamma, p)$: $[\tilde{\Sigma}_i] \cdot e_j$ is the multiplicity of the point $p_j$ on $\Gamma_i$. The same is true for the 
filling $W'$ and the arrangement  $(\Gamma', p')$.  Since the map $\phi_*: H_2(W \cup U) \to H_2(W' \cup U)$ sends the classes  $[\tilde{\Sigma}_i]$ to  the corresponding classes $[\tilde{\Sigma}'_i]$, it follows that the inidence matrices  $\mathcal{I}(\Gamma, \{ p_j\})$ and $\mathcal{I}(\Gamma', \{ p_j'\})$ are identical (up to labeling of the columns).   
\end{proof}
	
To distinguish fillings in the absolute sense (without fixing the markings), we appeal to 
Corollary~\ref{cor:diff}: the condition that a homeomorphism should preserve the marking of the boundary is automatic if the corresponding plumbing graph has no symmetries.  
	
	\begin{cor}\label{cor:nosymmetry}
		Suppose the minimal normal crossing resolution of a sandwiched singularity $(X,0)$ is given by a plumbing graph with no non-trivial automorphisms.  Let $(W,\omega)$ and $(W',\omega')$ be two minimal symplectic fillings of the contact link $(Y,\xi)$. If $W$ and $W'$ are homeomorphic, the incidence matrices of the corresponding DJVS immersed disk arrangements are equivalent. 
	\end{cor}

One can generalize the above statement slightly: it  holds whenever every automorphism of the plumbing graph fixes the vertices carrying the marking.

\subsection{Unexpected fillings}	
	
	We will say that an incidence matrix is realized by a picture deformation if there is a compatible decorated germ 
	$(C, w)$ and a picture deformation $(C^t, p^t)$  whose incidence relations for $t\neq 0$ are given by the matrix (up to a permutation of columns). A subarrangement of a DJVS arrangement  $(\Gamma, p)$ is a subcollection of immersed disks $\Gamma_{i_1}, \dots \Gamma_{i_r}$ with all of the marked points  
	$p_{j_1}, \dots, p_{j_s} \in p$ that lie on these disks. We have the following corollary to Proposition~\ref{incidence-matrix}. 
	
\begin{cor} \label{non-realizable} (1) Let $(\Gamma, p)$ be a DJVS immersed disk arrangement encoding the filling $W_{(\Gamma, p)}$ of a link of sandwiched singularity $(Y, \xi)$. If 
 $W_{(\Gamma, p)}$ is strongly homeomorphic to a Milnor fiber, then the incidence matrix  $\mathcal{I}(\Gamma, p)$
 is realized by a picture deformation. 
 
 (2) Suppose that $(\Gamma, p)$  has a DJVS subarrangement   $(\Gamma_{i_1}, \dots \Gamma_{i_r}, p_{j_1}, \dots, p_{j_s})$ whose incidence matrix is not realizable by a picture deformation. Then 
 $W_{(\Gamma, p)}$ is an unexpected filling. 
\end{cor}
	
	By Corollary~\ref{cor:nosymmetry}, the word ``strongly'' can be omitted if the plumbing graph for $Y$ has no symmetries. 
 
\begin{proof}
Suppose that  $W_M$  is a  Milnor fiber for a smoothing of some surface singularity with the link~$Y$; let $(C, w)$ be the decorated germ encoding the (analytic type of) the sandwiched singularity.  By the de Jong--van Straten theory,  $W_M$ corresponds to a picture deformation $(C^t, q)$ of $(C, w)$, with the collection of marked points $q$. 
By Remark~\ref{rmk:top-vs-an}, the DJVS arrangement  $(\Gamma, p)$ encoding the filling $W_{(\Gamma, p)}$ is compatible with the germ $(C, w)$. 
By Theorem~\ref{incidence-matrix}, if  $W_{(\Gamma, p)}$  is strongly homeomorphic to $W_M$, 
the incidence matrices  $\mathcal{I}(\Gamma, p)$ and  $\mathcal{I}(C^t, q)$ are the same, up to permutation of columns.

For the second statement, observe that any subarrangement of a picture deformation can also be realized as a picture deformation, by simply restricting the deformation to the corresponding subset of germ components. Thus if the subarrangement cannot be realized by a picture deformation, the larger arrangement cannot either.
\end{proof}

In \cite[Section 7]{PS}, we studied unexpected fillings of the simplest type of sandwiched singularities, those with reduced fundamental cycle, and found a necessary condition for certain disk arrangements to be realizable by picture deformations.  	In particular, we found examples of disk arrangements that are not realizable (Figures 15, 17 and their generalizations in \cite[Section 7]{PS}). Here is the upshot of our previous work, stated in a slightly modified form for brevity:

\begin{theorem} \label{thm:old} Let $(C, w)$ be a decorated germ such that $C$ is a pencil of $m \geq 10$ lines, and the weight $w = (w_1, \dots w_{m})$ with $w_i\geq m-1$ for all $i$. Then there exists a DJVS disk arrangement $(\Gamma, p)$ which is compatible with $(C, w)$ and not realizable by a picture deformation. 

Accordingly, suppose that $Y$ is a Seifert fibered space whose star-shaped plumbing graph has $m\geq 10$ legs, with the length of each leg is at least $m-1$. The central vertex has self-intersection $-m-1$, 
and all other vertices have self-intersection $-2$. Then $Y$ with its canonical contact structure as a link of singularity admits an unexpected Stein filling.  

Further, if $m=2N+5$  for $N\geq 4$, then every Seifert space as above admits at least $N+1$ pairwise non-homeomorphic Stein fillings. These fillings correspond to DJVS disk arrangements of $2N+5$ disks, not realizable by picture deformations and distingushed from one another by their incidence matrices. 
\end{theorem}

Together with Theorem~\ref{thm:old}, Corollary~\ref{non-realizable} allows us to produce large families of sandwiched singularities with unexpected fillings, extending the results of~\cite{PS}:

\begin{cor} \label{cor:include10lines} Suppose that $C = \cup_{i=1}^m C_i$ is a germ of a singular plane curve at $0$ such that there are $m\geq 13$ lines with distinct slopes among the irredicuble components $C_i$ of $C$. Then the link of the sandwiched singularity defined by $(C, w)$  admits at least $\frac{1}{2}(m-5$)  unexpected Stein fillings if  the weights
$w_i$, $i=1, \dots, m$  are sufficiently large.  
\end{cor}

\begin{proof} The line components of $C$ can be used to create  DJVS arrangements 
as in Theorem~\ref{thm:old} that are not realizable by any picture deformation.   Each of these arrangements, together with arbitrary picture deformations of the remaining components of $C$, gives a DJVS arrangement compatible with $(C, w)$ that encodes an unexpected filling by Theorem~\ref{non-realizable}. The hypothesis on sufficiently large weights is used to make sure that there are enough marked points for all interesection points in the resulting arrangement. All the arrangements will produce pairwise distinct fillings because they have subarrangements with distinct incidence matrices. 
\end{proof}

\begin{proof}[Proof of Theorem~\ref{thm2-intro}] Fix $N>0$, set $m=2N+5$ and let $S$ be a star-shaped plumbing graph as in 
Corollary \ref{cor:include10lines}. The graph $S$ has $m$ legs where every vertex has self-intersection
$-2$, the central vertex $v^*$ has self-intersection $-m-1$. If $S$ is augmented to a graph $S'$ by adding  $(-1)$ vertices on the ends of all the legs,  the corresponding plane curve germ $C_S$ consists of $m$ lines. 
	Choose an arbitrary sandwiched augmentation $G'$ of $G$ by $(-1)$ vertices and let $\mathcal{C}_G$ be the plane curve germ
	associated with that augmentation. Let $E_0$ be the last vertex to be blown down when the augmented graph $G'$ is contracted to a smooth point, as in Definition~\ref{def:EC}. We build a new 
 plumbing graph $R$ out of $G$ and $S$: we decrease the central vertex self-intersection of $S$ by 
$1$ (so that  now $v^*\cdot v^*= -m -2$), keep self-intersections of all the other vertices of $S$ and $G$ the same, and connect $v^*$ in $S$ to $E_0$ in $G$ by an edge.  Then $R$ is a sandwiched graph: augment it to a graph $R'$ 
by adding the $(-1)$ vertices that augment $G$ to $G'$, together with the $(-1)$ vertices on the end of each leg in $S$.  The resulting graph $R'$ can be blown down to a smooth point, by 
following the blow-down procedure for $G'$ and  blowing down all the legs of $S$. At the penultimate stage, 
we have two vertices connected by an edge, $E_0$ with self-intersection $-1$ (equipped with the curvettas from $G$) and $v^*$ with self-intersection $-2$ (equipped with the curvettas from $S$); 
we blow down first $E_0$, then $v^*$.  It follows that the curve germ $C_R$ for $R$ contains $C_S$ as a subgerm.   
The remaining irreducible components of $C_R$ are in 1-1 correspondence with the components of $C_G$ since they come from the curvettas on $E_0$ after an additional blowdown (but the resulting subgerm ${C}_G^*$ of $C_R$ is ``more singular'' than $C_G$ because of the additional blowdown).   As always, the weights of the curvettas in 
the decorated germ are controlled by the blowdowns and are determined by the graph with its augmentation. 

We would like to construct an unexpected DJVS arrangement $\Gamma$ compatible with the decorated germ $C_R$ with its weights. 
Since $C_R$ contains the $m$ lines forming $C_S$ as a subarrangement, we can take the union of the DJVS arrangement $\Gamma_S$ of Theorem~\ref{thm:old} and an arbitrary picture deformation $\Gamma_{G}^*$ of the complementary subgerm  $C^*_G$ (e.g. its Scott deformation), making sure that each pair of disks from two subgerms $C_S$ and $C^*_G$ intersect transversely at generic double points. In terms of braided wiring diagrams, see Figure~\ref{fig:wiringsubgraph}. Note that we run into a problem with weights here:  all intersection points of the resulting arrangement 
$\Gamma$ need to be decorated with marked points, but there are a lot of intersection points between disks from $C_S$ and $C_G^*$, and the weights induced by the augmented graph $R$ may not be sufficiently high. To rectify this, we modify the graph $R$ to increase  the weights of all curvettas 
without changing the germ $C_R$. We fix the arrangement $\Gamma$ that we constructed, mark all intersections between its disks, and let $w_{max}$ to be maximal number of marked points (counted with multiplicity) on any of the disks $\Gamma_i$ in $\Gamma$.  
Now, starting with the given augmentation of $R$, 
replace all the $(-1)$ vertices by sufficiently long 
chains of $-2$ vertices to make a graph $K$. Note that $K$ still contains $G$ as a subgraph.  
We can augment the graph $K$ by putting $(-1)$ vertices at the end of each chain; the resulting curve germ $C_K$ is identical to $C_R$ but the weights increase by the quantities given by the length of the $(-2)$ chains (see Remark~\ref{rmk:-2legs}).   Therefore we can make the weight of each curvetta in $C_K$ greater than $w_{max}$, so that the DJVS arrangement we constructed (possibly with additional free marked points) is compatible with the decorated germ $C_K$.   

\begin{figure}
	\centering
	\includegraphics[scale=.5]{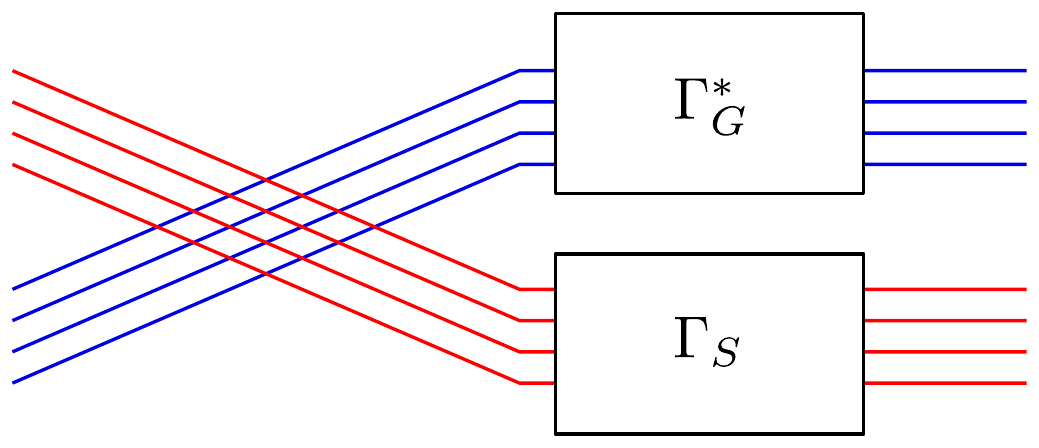}
	\caption{The DJVS arrangement $\Gamma$ compatible with the decorated germ $C_K$ built from the union of DJVS arrangements $\Gamma_G^*$ and $\Gamma_S$ in the Proof of Theorem~\ref{thm2-intro}. Each branch of $\Gamma_S$ intersects each branch of $\Gamma_G^*$ at a transverse double point. Thus each component of $\Gamma_S$ (which has multiplicity $1$) has $d_i$ transverse intersections with a component of $\Gamma_G^*$ whose multiplicity is $d_i$.}
	\label{fig:wiringsubgraph}
\end{figure}

\begin{figure}
	\centering
	\includegraphics[scale=.5]{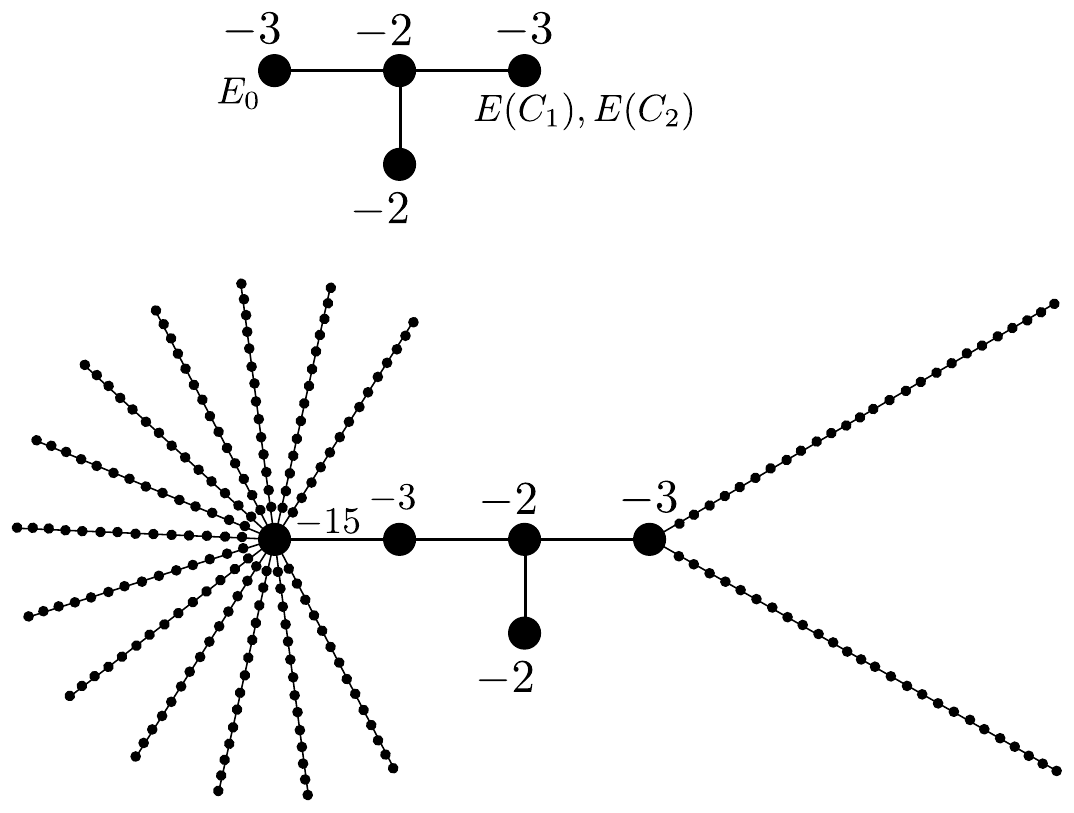}
	\caption{Above: The sandwiched dual graph $G$ from Figure~\ref{fig:examplesandwich}. Below: A sandwiched dual graph $K$ containing $G$. Unlabeled small vertices have self-intersection $-2$. The link of the sandwiched singularity with normal crossing resolution $K$ admits at least $5$ pairwise non-homeomorphic unexpected Stein fillings. As in the proof of Theorem~\ref{thm2-intro} with $N=4$, $K$ is obtained from $G$ by adding a vertex $v^*$ to the left-most $-3$ vertex associated with $E_0$ such that the self-intersection on $v^*$ is $-15$, with $13$ additional legs attached to $v^*$ of length $14$ with all $-2$-vertices, and $2$ additional legs, each of length $26$ and with all $-2$-vertices, attached to the $-3$ vertex associated with $E(C_1)=E(C_2)$ on the far right.}
	\label{fig:unexpectedsubgraph}
\end{figure}

\end{proof}

\begin{remark} It is not hard to give an estimate of $w_{max}$ in terms of the number of vertices in the graphs $G$ and $S$ and the multiplicities of the curvettas of $C_G$ at $0$: each of the lines corresponding to $S$ intersects a curvetta of multiplicity $d$ at $d$ generic points. It is also possible to be more strategic with the construction of arrangement if one wants to build a graph $K$ with a smaller number of the $(-2)$ vertices by tailoring the pairwise intersections between components of $C_S$ and $C'_G$ to occur at already marked singular points of their DJVS subarrangements. 
\end{remark}

\begin{remark} \label{rmk:insideout} The incidence matrix  $\mathcal{I}(\Gamma, p)$ encodes the homological data of the vanishing cycles of the nearly Lefschetz fibration $W_{(\Gamma, p)}$.  In the simplest case of classical open books and fibrations (or equivalently, decorated germs with smooth components),  each disk of $\Gamma$ corresponds to a single hole of 
the disk-with-holes fiber $P$. Then, an interesection point of several components of $\Gamma$ corresponds to a vanishing cycle that encloses the corresponding holes in the fiber: this fixes the homology class of the vanishing cycle but not its isotopy class. 
An incidence matrix   $\mathcal{I}(\Gamma, p)$ is realizable by a picture deformation if and only if the Lefschetz fibration on  $W_{(\Gamma, p)}$ has the vanishing cycles with the same homological data as the Lefschetz fibration on a Milnor fiber. (For the ``if'' part, represent the Milnor fiber as a picture deformation of a compatible germ, which will provide a 
Lefschetz fibration whose fiber is identified with that of the fibration $W_{(\Gamma, p)}$.)
It is interesting to note that in the definition of $\mathcal{I}(\Gamma, p)$ and in the construction of the cap of a filling, the outer boundary component of the open book plays a special role. The world of Lefschetz fibrations is more symmetric: the fiber is a sphere with holes with no distinguished boundary component. Stating which holes are enclosed 
by a curve requires distinguishing ``the inside'' and ``the outside'' of the curve, so one needs to fix the choice the outer boundary component of the disk with holes; however, 
the homology class of the curve on  the punctured sphere is well defined without this choice. There is an obvious way to translate the data between different choices of the outer boundary component when the fiber is ``turned inside out'' (see Proposition \ref{prop:insideout} below for a specific example). Making a different choice of the outer boundary produces a DJVS arrangement $(\Gamma', p')$ different from $(\Gamma,p)$: the arrangement $(\Gamma',p')$ will be compatible with a decorated germ $(C', w')$ given by a different combinatorial choice of the augmented resolution graph, since the outer component corresponds to the vertex $E_0$, see \cite[Lemma 4.3]{PS} and Section~\ref{sandwich-setup}. The previous discussion implies that if one of the arrangements  $(\Gamma, p)$ and $(\Gamma', p')$ has incidence matrix  realizable by the picture deformation, then the other also does, because the Lefschetz fibrations  $W_{(\Gamma, p)}$ and     $W_{(\Gamma', p')}$ are identical. 

A similar statement is true for the general spinal case, but one can only turn the open book inside out if there is another boundary compoment of multiplicity 1. We leave the details to the reader.
\end{remark}

We can use the previous remark to generate a new kind of a non-realizable arrangement. 
All of our unexpected fillings above are using the same type of examples as a key ingredient: we need a collection of lines with distinct slopes to be a subgerm of the given decorated germ. By contrast, Proposition~\ref{prop:insideout} gives examples where all the curvettas have a high-order tangency with one another. The construction is based on a trick from Remark~\ref{rmk:insideout} and doesn't produce any new unexpected Stein fillings, these examples shed some more light on arrangements {combinatorially} non-realizable by picture deformations. We give both a topological proof of Proposition~\ref{prop:insideout} and an analytic explanation of same fact in Proposition~\ref{prop:analytic-insideout}.

	\begin{figure}
		\centering
		\includegraphics[scale=.5]{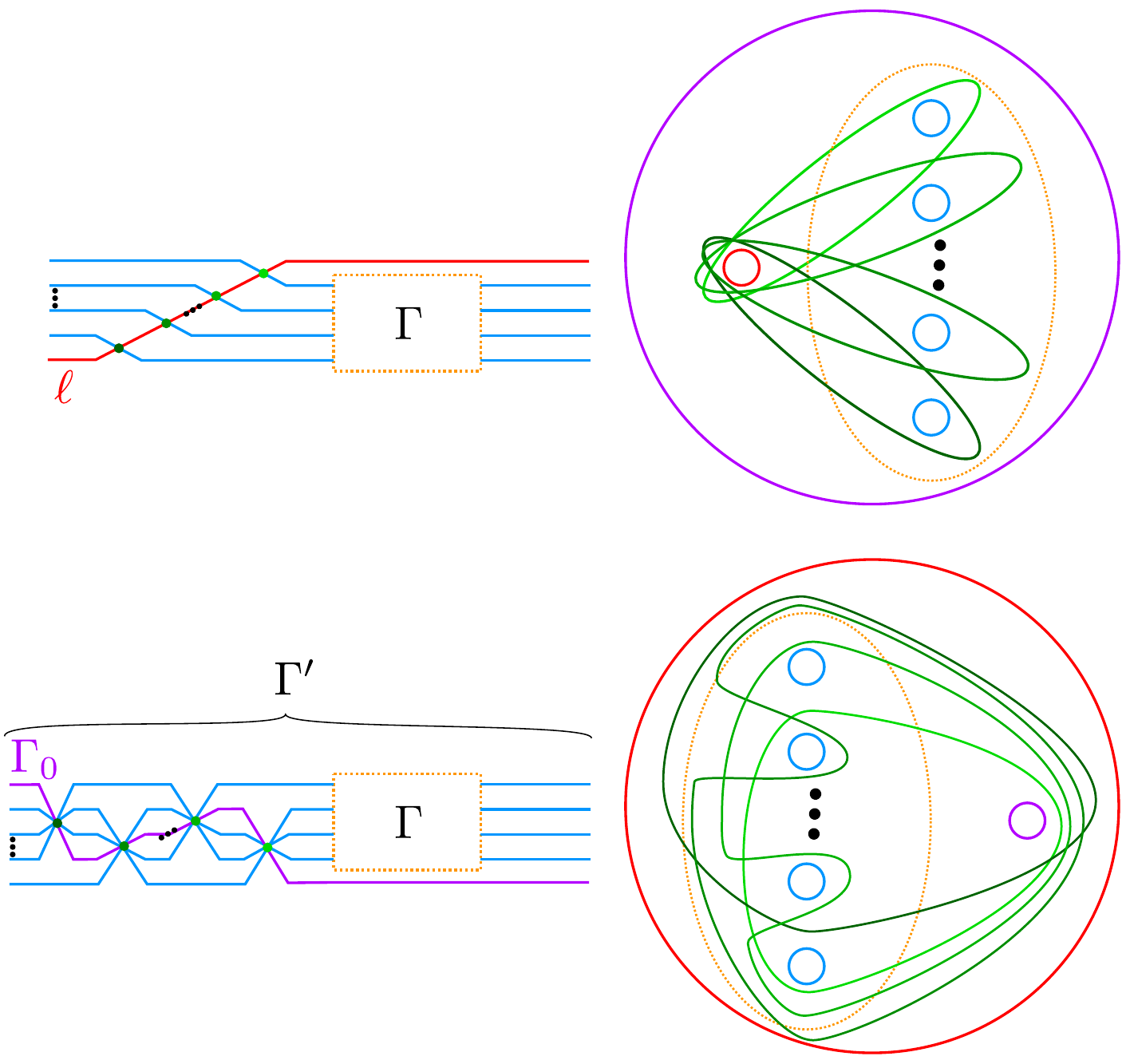}
		\caption{Wiring diagrams for $\Gamma\cup \ell$ and the inside-out version $\Gamma'$ along with a fiber for each showing the corresponding vanishing cycles.}
		\label{fig:insideout}
	\end{figure}

	\begin{prop}\label{prop:insideout} Let $(X, 0)$ be a sandwiched singularity with a decorated germ  $(C, w)$, and $(\Gamma, p)$ a DJVS arrangement that is compatible with $(C, w)$. 
	{Suppose that the incidence matrix $\mathcal{I}(\Gamma, p)$
	is not realizable by a picture deformation.} Then for any fixed integer $r>0$,  there exists a DJVS arrangement $(\Gamma', p')$ induced by $(\Gamma, p)$ and compatible with another decorated germ $(C', w')$, such that all components of $C'$ are tangent to one another up to order $r$, and the matrix 
	{$\mathcal{I}(\Gamma', p')$ is not realizable by a picture deformation.}
\end{prop}
	\begin{proof}  For simplicity we assume 
	that the germ $(C, w)$ consists of $m$ {\em smooth} components, so $(\Gamma, p)$ is compatible with a classical open book, where each component $\Gamma_i$ of $\Gamma$ corresponds 
	to one hole $h_i$  in the planar page, $i=1, \dots, m$.  (With minimal modifications, the same proof
	works for the general case as well.) We  also assume $m>2$.  
	
	We construct $(\Gamma', p')$ from $(\Gamma, p)$.  Add one complex line $l$  that intersects all components of $\Gamma$ at generic double points, and mark all new  intersections in addition to the marked points in $p$ to get the decoration $p_{\Gamma \cup l}$. 
	We can assume that the projection of these $m$ extra intersections to $D_x$ lies is a disk disjoint from the disk containing the 
	intersection points of $\Gamma$, for example by working with the wiring diagrams (we can assume that the extra intersections in the wiring diagram precede the rest of the arrangement). It follows that $(\Gamma \cup l,  p_{\Gamma \cup l} )$ is compatible 
	with the decorated germ $(C\cup l, w_{C \cup l})$, where $l$ is a 
	line transverse to all other components of $C$, and the weight $w_{C \cup l}$ is given by 
	$w_{C \cup l} (C_i)= w(C_i)+1$ for each component $C_i$ of $C$,  $w_{C \cup l}(l)= m$.
	Let $(X', 0)$ denote the corresponding singularity; the Lefschetz fibration 
	$W_{(\Gamma \cup l,  p_{\Gamma \cup l})}$ is then a filling of its contact link.
	Since the arrangement $(\Gamma \cup l,  p_{\Gamma \cup l})$ has a   subarrangement $(\Gamma, p)$ with a non-realizable incidence matrix, it follows that the incidence matrix  
	$\mathcal{I}(\Gamma \cup l,  p_{\Gamma \cup l})$ is not realizable by a picture deformation.

	We examine the vanishing cycles of the corresponding Lefschetz fibrations.
	The planar open book factorization induced  by  $(\Gamma \cup l,  p_{\Gamma \cup l})$  is related to that for $(\Gamma, p)$ as follows: the page of the open book gets a new hole $h_l$ corresponding to $l$, and the monodromy is composed with a product of positive Dehn twists, each of which encloses $h_l$ and one of the holes $h_i$ corresponding to components of $\Gamma$. 
	
	Now, we turn the page of this open book inside out by swapping the role of $\d h_l$ and the outer boundary component of the open book. In the new planar page identification with a disk with holes, instead of $h_l$ there is a hole $h_o$ corresponding to the old outer boundary component; $\d h_l$ becomes the new outer boundary. Let $(\Gamma', p')$ be the DJVS arrangement encoded by this new open book presentation. By Remark~\ref{rmk:insideout}, the incidence matrix  $\mathcal{I}(\Gamma', p')$ is not realizable by a picture deformation because we know that the matrix 
	$\mathcal{I} (\Gamma \cup l,  p_{\Gamma \cup l})$ is not realizable.

	It is easy to see how $\Gamma'$ is related to $\Gamma\cup l$ and $\Gamma$.  In the new open book factorization, the Dehn twists corresponding to intersection points between components of $\Gamma_i$ remain unchanged and are in the obvious 1-1 correspondence with the Dehn twists in the old factorization. What changes is the Dehn twists involving the extra hole: a Dehn twist enclosing $h_l$ and $h_i$ now becomes a Dehn twist in the new open book that encloses $h_o$ and all of the $h_j$'s {\em except} $h_i$. By our assumptions, these Dehn twists are precomposed with the rest of the monodromy, and the arrangement $\Gamma'$ is build by adding a disk $\Gamma_0$ to $\Gamma$ and precomposing with extra wiring with transverse intersection points  of $m$-tuples $(\Gamma_0, \Gamma_1, \dots, \hat{\Gamma}_i, \dots, \Gamma_m)$, in order. (Here, the hat indicates that the corresponding disk is not included.)   
	It is easiest to visualize this wiring diagram by placing the intersection points of $\Gamma_i$ with the $\C$-fiber of the diagram at roots of unity for $i=1, \dots, m$, with $\Gamma_0$ at the origin, and making all strands except one meet at the origin, in order. See Figure~\ref{fig:insideout}.
	
	For any two components of $(\Gamma', p')$ other than $\Gamma_0$, there are  $m-2>0$ additional transverse intersection points (outside of the part corresponding to $\Gamma$); $\Gamma_0$ intersects any other disk $m-1$ times. It follows that the order of tangency between any two components is at least $m-1$. (This can also be seen by using the ``daisy relation'' \cite{HMVHM} for the factorization to see $m-2$ parallel Dehn twists around the boundary component $\d h_l$.) To get a the given order of tangency $r$, we can repeat the above procedure (at the cost of adding an extra line several times, as needed). 
	\end{proof}

The inside-out topological  argument above shows that the arrangement $\Gamma'$ is not realizable by a small analytic deformation if the arrangement $\Gamma$ is not. We are able to give a direct analytic proof of this fact. We will assume that the components of $\Gamma$ are smoothly {\em embedded} graphical disks (equivalently, the compatible germ $C$ has smooth components). Then a small deformation of this germ also has smooth graphical components.
The weights and marked points do not play a role here, and we omit them from the statement below. (The incidence matrix records intersections between the components of $\Gamma$.) Note that the proposition below is stronger than what the topological argument implies: we only need a combinatorial assumption on the incidences of $\Gamma'$, without any specific requirements on the wiring diagram.  

\begin{prop} \label{prop:analytic-insideout} Suppose that $\Gamma$ a smooth disk arrangement
compatible with a germ $(C, w)$ with smooth components, such that the incidence matrix $\mathcal{I}(\Gamma)$ is not realizable by a small analytic deformation.  Consider an arrangement $\Gamma'$ related to $\Gamma$ as in Figure~\ref{fig:insideout}: $\Gamma'$ has an extra disk $\Gamma_0$ in addition to the components of $\Gamma$, and additional transverse intersection points  of $m$-tuples $(\Gamma_0, \Gamma_1, \dots, \hat{\Gamma}_i, \dots, \Gamma_m)$. Then the incidence matrix of the arrangement $\Gamma'$ is not realizable by a small analytic deformation of a compatible germ. 
\end{prop}

\begin{proof} Assume the contrary, namely, there are equations of the form $y = f_j(x, s)$ with analytic functions $f_j(x,s)$ that encode a disk arrangement with the incidences of $\Gamma'$ for $s>0$, specializing to a compatible curve germ for $s=0$. Substracting the equation $y = f_0 (x, s)$ of the disk $\Gamma_0$ from each of the equations, we can assume that $\Gamma_0$ is given by $y=0$. Then for each $s>0$, the  $m$-tuple intersections $(\Gamma_0, \Gamma_1, \dots, \hat{\Gamma}_i, \dots, \Gamma_m)$ are points $(x=r_i(s), y=0)$ for some analytic functions $r_i(s)$, and the components 
$\Gamma'_1, \dots, \Gamma_m'$ are given by equations of the form $y = q_i(x, s)$, where 
$$
q_i(x, s)= (x - r_1(s))\dots (x-r_{i-1}(s))(x-r_{i+1}(s)) \dots (x-r_m(s)) g_i(x, s).
$$
Here,    $g_i(x, s)$ are analytic functions that are non-vanishing for all $(x, s)$: by assumption, the only places where $y=q_i(x,s)$ intersects $y=0$ is at $r_1(s),\cdots, r_m(s)$, and these zeros have multiplicity one because the intersections with $y=0$ are transverse.

Let
$$
p(x, s) = (x-r_1(s)) (x-r_2(s)) \dots (x-r_m(s)), 
$$
and consider, for $s >0$, an arrangement given by equations 
$$
y = \frac{p(x, s)}{q_i(x, s)} = \frac{(x-r_i(s))} {g_i(x, s)}.
$$
 For $s=0$, these equations specialize to equations of a reducible curve germ at $0$. For $s>0$, the incidences of this new arrangement are the same as the incidences of $\Gamma$: indeed, the above formulas produce the incidences of
 the family $y = q_i(x, s)$ with all the points $(r_i(s), 0)$ removed. In both cases, incidences between the $i^{th}$ and $j^{th}$ components correspond exactly to $(x,s)$ solutions to $(x-r_i(s))g_j(x,s)=(x-r_j(s))g_i(x,s)$. It follows that if the combinatorial arrangement $\Gamma'$ admitted a realization by a small analytic deformation, the same would be true for $\Gamma$, a contradiction. 
\end{proof}
Since we work directly with equations of the curves, the above proposition and its proof can be restated for deformation realizations of arrangements up to isotopy (rather than in the combinatorial sense).

\begin{remark} Although it doesn't seem to lead to interesting topological consequences, there is a more obvious way to create deformation nonrealizable arrangements compatible with germs with smooth components tangent to a high order. We can simply use a known unrealizable arrangement and add a number of points where all the components intersect.  Indeed, let $\Gamma$ be an arrangement with $m$ smooth components  as in Proposition~\ref{prop:insideout}. 
Consider an arrangement $\Gamma'$, also with $m$ smooth components and singularities given by transverse multipoints only, such that the incidence relations of $\Gamma'$ are given by those of $\Gamma$, plus $k$ additional $m$-points. Suppose that $\Gamma'$ can be obtained by a small analytic deformation. As before, we can assume that one of the components of $\Gamma'$ is given by the line $y=0$ at all times. If $(x=r_i(s), y=0)$ are the $m$-points for $i=1, \dots, k$, then the deformation for the curves producing $\Gamma'$ is given by equations of the form
$$
y= (x - r_1(s))\dots (x-r_k(s)) g_i(x, s).
$$
Here, $g_i(x,s)$ are analytic functions (the one corresponding to the curve $y=0$ is identically 0). It is easy to see that functions $g_i(x,s)$ give a germ deformation with incidences of $\Gamma$ for $s>0$, so if the incidence matrix of $\Gamma$ is not realizable by a small analytic deformation, neither is the incidence matrix of $\Gamma'$. 
\end{remark}

Using the last two constructions together with  \cite[Lemma 7.5]{PS} producing unexpected pseudoline configurations, we can create various non-realizable arrangements with smooth components. However, more general criteria and similar direct analytic proofs of deformation nonrealizability remain elusive.
We have no results for arrangements corresponding to cuspidal germs, in particular, 
no direct explanation for nonrealizability of arrangements of nodal cubics in~\cite{BPS}. It would be very interesting to develop further tools to 
detect interesting arrangements.

\bibliography{references} 
\bibliographystyle{alpha}

\end{document}